\documentclass[12pt,reqno]{amsart}
\usepackage{color}
\usepackage[title]{appendix}
\usepackage{latexsym}
\usepackage{amsmath,amssymb,dsfont,amsfonts,nccmath}
\usepackage{mathrsfs}
\usepackage{verbatim}
\usepackage{graphicx}
\usepackage{graphics}
\usepackage{geometry}
\usepackage{booktabs}
\usepackage[shortlabels]{enumitem}
\usepackage{xcolor}
\usepackage{blindtext}
\usepackage{fancyhdr}
\newtheorem{remark}{Remark}[section]
\newtheorem{proposition}[remark]{Proposition}
\newtheorem{theorem}[remark]{Theorem}
\newtheorem{definition}[remark]{Definition}
\newtheorem{corollary}[remark]{Corollary}
\newtheorem{lemma}[remark]{Lemma}
\newtheorem{hypo}[remark]{Hypothesis}
\def\R{\mathbb R}
\def\N{\mathbb N}

\def\E{\mathbb E}
\def\P{\mathbb P}

\def\shb{{\mathcal B}}

\def\shd{{\mathcal D}}
\def\shf{{\mathcal F}}
\def\shh{{\mathcal H}}

\def\partialt{{\partial_t}}

\newenvironment{prooff}{{\bf \textit{Proof}}}{\hfill $\Box$ \\}

\numberwithin{equation}{section}
\allowdisplaybreaks

\title{About semilinear low dimension Bessel PDE\MakeLowercase{s}}
\author{Alberto OHASHI$^1$}

\author{Francesco RUSSO$^2$}

\author{Alan TEIXEIRA\textsuperscript{3}}

\address{$1$ Departamento de Matem\'atica, Universidade de Bras\'ilia, 70910-900, Bras\'ilia, Brazil.} \email{amfohashi@gmail.com}

\address{$2$ ENSTA Paris, Institut Polytechnique de Paris,
	Unit\'e de Math\'ematiques appliqu\'ees, 828, boulevard des Mar\'echaux, F-91120 Palaiseau, France}
\email{francesco.russo@ensta-paris.fr}
\email{\textsuperscript{3}alanteixeiranicacio@yahoo.com.br}

\date{30th March 2024}

\begin{document}
\begin{abstract}
  We prove existence and uniqueness of 
  solutions of a semilinear PDE driven by a Bessel type generator
$L^\delta$  with low dimension $0 < \delta < 1$.
 $L^\delta$ is a local operator, whose drift is the
derivative of $x \mapsto \log (\vert x\vert)$:
in particular it is a Schwartz distribution, which
is not the derivative of a continuous function.
The solutions 
are intended in a duality {\it (weak)} sense
with respect to state space
$L^2(\R_+, d\mu),$ $\mu$ being 
an invariant measure for the Bessel semigroup.

\end{abstract}

\maketitle

\maketitle



\section{Introduction}

In this paper, we investigate the existence and uniqueness of weak solutions of semilinear parabolic
partial differential equations (PDEs) driven by the generator $L^\delta$ of the Markov semigroup associated with the Bessel process with dimension  $0 < \delta <1$. 

Recall that the class of Bessel processes is a family of  Markov processes with values in $\R_+$ and parameterized by $\delta \in \mathbb{R}_+$, called the {\it dimension}. Throughout this paper, we refer the regime $\delta \in (0,1)$ as the \textit{low dimension} case. Bessel processes have been largely investigated in the literature. We refer the reader to e.g \cite{mansuy,zamb,Yor} for an overview on the basic theory. Typical examples of low dimension Bessel processes appear in queueing theory (see e.g. \cite{Coffman}) and the theory of Schramm-Loewner evolution curves, see e.g. \cite{lawler}. Two-parameter family of Schramm-Loewner evolution $\textit{SLE}(\kappa,\kappa-4)$ defined in \cite{werner} provides a source of examples of Bessel flows with very singular behavior when $\delta = 1-\frac{4}{\kappa}, \kappa> 4$. We refer the reader to \cite{dubedat, beliaev2020new} for more details. Bessel processes with dimension $\delta>1$ also naturally appears in interacting particle systems related to random matrix theory (see e.g. \cite{katori}) and in stochastic volatility or interest rates models in mathematical finance, see e.g. \cite{Timer}.

The dimension $\delta\in \mathbb{R}_+$ dictates the regularity of the Bessel process. Indeed, 
in the regular case $\delta > 1$, it is a pathwise non-negative
solution of
\begin{equation}\label{NMBES}
dX_t = \frac{\delta-1}{2} X^{-1}_tdt + dW_t,\quad X_0=x_0,
\end{equation}
where $W$ is a standard Brownian motion,
see for instance Exercise (1.26) of Chapter IX in \cite{Yor}. In particular, $X$ is an It\^o process.
In the low dimension case $\delta \in(0, 1)$, the integral $\int_0^t X^{-1}_sds$ does not converge and the correspondent Bessel process is not a semimartingale but a Dirichlet process, see e.g. \cite{FolDir} or Chapter 14 of
\cite{Russo_Vallois_Book}. In fact, in that case, the Bessel process is highly recurrent and hits
zero very often, so that $x_0 = 0$ is a true singularity
for the Bessel process. Consequently, in the low dimension case $\delta \in(0, 1)$, we cannot expect
it to be a solution of a classical stochastic differential equation (SDE).

It is a well-established fact that any low-dimension Bessel process can be decomposed into the sum of a standard Brownian motion and a null quadratic variation process, expressed in terms of a density occupation measure via local times. For further insights, one may refer to works such as \cite{BertoinBessel} and \cite{mansuy}. Recently, in the study conducted by \cite{ORT1_Bessel}, the authors characterize any low dimension Bessel process as the unique solution of \eqref{NMBES} interpreted as a Stochastic Differential Equation (SDE) with a distributional drift. In this novel perspective, the singular drift, represented by $x \mapsto \frac{\delta-1}{2}\frac{1}{x}$, is elucidated as the derivative, in the sense of Schwartz distributions, of the function $x \mapsto \frac{\delta-1}{2}{\rm log} \vert x \vert$, which complements earlier representations based on principal values through local times. The authors prove that the Bessel process with dimension $\delta \in (0,1)$ is the unique non-negative solution of a suitable strong-martingale problem driven by its generator $L^\delta$. We refer the reader to \cite{ORT1_Bessel} and Definition \ref{smp} for further details.

In this article, we devote our attention to the semilinear backward Kolmogorov PDE associated with the generator $L^\delta$ of the Bessel process in the singular regime $\delta \in (0,1)$ 
\begin{equation}\label{pde01}
	\left\{
	\begin{array}{l}
		(\partialt + L^\delta) u + f(\cdot, \cdot, u, \partial_x u)=0\\
		u(T,x) = g(x).
	\end{array}
	\right.
\end{equation}
Formally speaking 
\begin{equation}\label{Ldeltades} 
L^\delta \phi = \frac{\phi''}{2} + \frac{\delta - 1}{2} p.v. \frac{1}{x}
\phi',
\end{equation}
where $p.v.$ is the Cauchy principal value.
In \eqref{pde01}, the coefficient 
$f:[0,T]\times\R_+\times \R\times \R \rightarrow \R$ is
a Borel Lipschitz function in the two last variables
and
$g$ is a Borel function
belonging to $L^2(d\mu),$
where $\mu$ is the  Borel measure
described in Definition \ref{Defmu}. 
If $f = 0$, then the candidate solution is
given by
$u(s,x_0) = \E(g(X^{x_0}_{T-s})), s \in [0,T]$,
where $(X^{x_0}_t)$ is the standard Bessel process starting at $x_0$
with dimension $\delta \in (0,1)$.

As described in \cite{ORT1_Bessel}, $L^\delta$ is a particular example of generator
of an SDE with distributional drift,
\begin{equation}\label{SDE}
dX_t = dW_t + b(X_t) dt,
\end{equation}
where the drift $b$ is a distribution (not a measure). Recently, several authors have been investigated generators associated with SDEs of the type \eqref{SDE}. In this direction, we refer the reader to e.g. \cite{issoglio_russoMPb} and other references therein.
In that  literature, in general, $b$ is the derivative (gradient)
of a continuous function.
In the present work, the drift in the description of the action $L^\delta \phi$ in (\ref{Ldeltades}) is the derivative of a discontinuous function
unbounded at zero. As in typical examples of SDEs with distributional
drift, the domain of the generator $\shd_{L^\delta}(\R)$ 
does not contain all the smooth functions (even with compact support).
In particular, the function $\phi(x) = x$ even cut outside
a compact interval does not belong to  $\shd_{L^\delta}(\R)$.

The main result of the present article is Theorem \ref{MainT}, which
states the existence and uniqueness of a {\it weak solution} of  the PDE \eqref{pde01},
where the duality is described by
a space of test functions $D$
suitably specified in \eqref{SetD} and the pivot space is $L^2(d\mu),$ see Definition
\ref{D42}.
Proposition \ref{WeakMild}
 reduces the problem to the existence and uniqueness
of mild solutions (see Definition \ref{D54}), where the corresponding semigroup
$(P_t^\delta)$ is the one associated with the generator
$L^\delta.$
By Proposition \ref{c1Bis}, the Borel measure $\mu$ is an invariant measure,
analogously to the case of heat semigroup (associated with Brownian motion),
where this role is played by the Lebesgue measure.

The proof of Theorem \ref{MainT} relies on a suitable fixed point
theorem based on the crucial inequality
stated in Proposition \ref{uti}, which plays the role of
some basic Schauder estimate.
The natural evolution space for \eqref{pde01}
is the space $\shh$ of elements in $L^2(d\mu)$ whose
derivative in the sense of distributions belongs again
to $L^2(d\mu)$.
Lemma \ref{l1} shows that the space $\shh$ fits with the domain
of a symmetric closed form, defined
on the Hilbert space $H = L^2(d\mu),$
which is equal to the domain of the square root of the map
$-L^\delta_F,$  the so called
 (self-adjoint) {\it Friederichs extension}
of the operator $-L^\delta$, see Proposition \ref{timecont}.

This article is organized as follows. Section \ref{preres} presents some preliminary results concerning the Bessel process and its associated semigroup. Section \ref{Fext} presents Friedrichs self-adjoint extension of the symmetric positive linear operator $-L^\delta$. Sections \ref{linearSEC} and \ref{NLPDE} discuss the main results of the paper. In particular, Theorem \ref{MainT} given in Section \ref{NLPDE} describes the existence of a unique weak solution of the PDE (\ref{pde01}) under mild regularity conditions (Hypotheses \ref{Hypf} and \ref{Hypg}) on the coefficient $f$ and the terminal data $g$. The appendix \ref{apA} briefly recalls some basic results about the Friedrichs self-adjoint extension. Appendix \ref{apB} presents the proofs of Proposition \ref{eqH}, Lemma \ref{l1} and other technical auxiliary results.  

\section{Recalls and preliminary results}

\label{preres}

\subsection{Bessel process
  and strong martingale problem}

$ $

Given $0< \delta < 1$, we introduce our basic Borel
$\sigma-$finite positive measure on $\R_+$.
\begin{definition} \label{Defmu}
  $$\mu(dx) := x^{\delta - 1} \mathds{1}_{\{x > 0\}} dx.$$
   \end{definition}
  We call $L^2(d\mu)$ as the space of functions
$f:\R_+\rightarrow \R$ such that $\int_{\R_+}f^2(x)\mu(dx) < \infty$. The inner product and norm of the Hilbert space $L^2(d\mu)$ will be denoted, respectively, by $\langle\cdot,\cdot\rangle_{L^2(d\mu)}$ and $||\cdot||_{L^2(d\mu)}$.
Throughout this article, we denote by $L^{1,2}(dt d\mu)$ as the Hilbert space
$L^1([0,T];L^2(d\mu))$.
The usual inner product on $L^2(\R_+,dx)$ and related norm
are denoted by $\langle\cdot,\cdot\rangle$ and $||\cdot||$ respectively.
In the paper, for a given function $v:[0,T] \times \R_+ \rightarrow \R$, we will often denote $v(t):=v(t,\cdot), t \in [0,T]$.

Similarly as in Definition 3.2 and considerations before
Remark 3.5 in \cite{ORT1_Bessel}, we set
\begin{equation} \label{DL2R+}
  \shd_{L^\delta}(\R_+):= \{f \in C^2(\R_+) \vert f'(0) = 0\},
  \end{equation}
\begin{equation} \label{SetD}
 D :=\{\varphi \in C^2_0(\R_+);  \varphi'(0) = 0\},
\end{equation}
where $C^2_0(\R_+)$ is the space of $C^2$ functions on $\R_+$
with compact support. Notice that
$D$ is a subset of the set $\shd_{L^{\delta}}(\R_+)$
and $D\subset L^2(d\mu)$.
\begin{remark} \label{RSetD}
$D$ is a closed subspace of the Banach
space $C_b^2(\R_+)$ of bounded functions $f: \R_+ \rightarrow \R$
of class $C^2$ with first and second order bounded derivatives.
\end{remark}

We define
${L}^{\delta}: \shd_{L^{\delta}}(\R_+)  \rightarrow C(\R_+)$ as 
\begin{equation}\label{op}
	L^{\delta}f(x) :=\left\{
	\begin{array}{ll}
		\dfrac{f''(x)}{2} + \dfrac{(\delta - 1)f'(x)}{2x},& x > 0\\
		\delta f''(0),& x = 0.
	\end{array}
	\right.
\end{equation}

  \begin{remark}\label{r1}
	\
  \begin{enumerate}
  \item
    Except when explicitly stated otherwise,  $L^{\delta}$ with stand for the restriction to $D$ of \eqref{op}.
\item  If $f \in D$ and $x > 0$, we can write
  $L^{\delta}f(x)=\dfrac{id^{1-\delta}}{2}(id^{\delta-1}f')'(x)$, where $id$ stands for the identity function
  $x \mapsto x$.
  \end{enumerate}
\end{remark}

 Given a fixed Brownian motion $W$ on some fixed probability
 space $(\Omega, \shf, \P)$ and $x \geq 0$, the $\delta$-squared Bessel process $S$ is the unique positive strong solution of
\begin{equation}\label{sqbs}
	dS_t = \sqrt{S_t}dW_t + \delta dt, S_0 = x_0^2, t\geq 0.
\end{equation}
Let $ x_0 \in \R_+, \delta \in (0,1)$.
The $\delta$-dimensional Bessel process, $X$, starting from $x_0$ is characterized
as $X_t=\sqrt{S_t}$.

Let $s \in [0,T]$.
Obviously we can also consider the (unique) strong solution
$S = S^{s,x_0}$ of 
\begin{equation}\label{sqbs_times}
	dS_t = \sqrt{\vert S_t \vert}dW_t + \delta dt, S_s = x_0^2, t\geq s.
      \end{equation}
The process $S$ is necessarily non-negative because of comparison theorem. For details,
      see Proposition 2.18 in Chapter 5 of \cite{ks}.
      We set $X := X^{s,x_0}$ defined as  $X_t=\sqrt{S_t}.$
  \begin{definition}\label{smp}
Let $(\Omega, \mathcal{F}, \P)$ be a probability space and let
$ \mathfrak{F} = (\mathcal{F}_t)$ be the canonical filtration associated with a
fixed Brownian motion $W$.
Let $x_0 \in \R_+$.
We say that a continuous $\mathfrak{F}$-adapted $\R_+$-valued process $Y$
 is a {\bf solution to
  the strong martingale problem}
(related to $L^\delta$)
with respect to $D$
and $W$ (with related filtered probability space),
with initial condition $Y_s = x_0$ 
if
\begin{equation}\label{lmpBis}
	f(Y_t) - f(x_0) - \int_{s}^{t}L^\delta f(Y_r)dr =
	\int_{s}^{t}f'(Y_r)dW_r,
\end{equation}	
for all $f \in D$.
\end{definition}
\begin{remark} \label{Requivalence}
        The strong martingale problem formulation above is equivalent
to the  one, considered in  \cite{ORT1_Bessel}, where $D$
is replaced by $\shd_{L^\delta}(\R_+),$
see Proposition 3.6 of  \cite{ORT1_Bessel}.
This property can be established by density arguments and also by Remark \ref{rem:SMP} (1) in Appendix \ref{apB}.
\end{remark}

\begin{proposition} \label{PSMP}
  Let $x_0\geq 0$ and  $s \in [0,T)$. The process
  $Y = X^{s,x_0}$ is the unique solution of the strong martingale problem with respect to
(related to $L^\delta$) $D$ 
and $W$, with initial condition $Y_s = x_0$.
\end{proposition}
\begin{proof} \
  Without loss of generality, we can suppose
  $s = 0$.
The existence follows by
Proposition 3.6 of \cite{ORT1_Bessel} and
the fact that  $D \subset \shd_{L^\delta}(\R_+)$.
By Remark \ref{Requivalence},
a solution of the martingale problem with respect
to $D$ in the sense of Definition \ref{smp}
also fulfills the one where $D$ is replaced
by $\shd_{L^\delta}(\R_+)$.
At this point, uniqueness follows by
Proposition 3.11 of \cite{ORT1_Bessel}.
  \end{proof}

   For every $s \in [0,T)$,
  the distribution of the process $(X_t^{s,x_0})_{t \in [s,T]}$
is the same as $(X^{x_0}_{t \in [0,T-s]}),$ where $X^{x_0}$ is the classical Bessel
process starting at $x_0$. In other words, the
unique solution $Y$ mentioned in Proposition \ref{PSMP}
is a shifted Bessel process with starting point $x_0$
and dimension $\delta$.

For further details on the (strong) martingale problem, the reader may consult
\cite{ORT1_Bessel}, where we focused on the case $s = 0$, but we can easily adapt
to the present framework.
Here, we need to introduce a
time inhomogeneous version
of the strong martingale problem.

\begin{proposition}\label{nhMP}
 Let $s \in [0,T), x_0 \ge 0$. Then 
	$X:= X^{s,x_0}$ solves the (inhomogeneous) strong martingale problem
	with respect to $D$ and $W$, with initial condition $X_s = x_0$.
        This means the following: For every $u \in C^{1,2}([s,T] \times \R_+;\R_+)$
  such that $\partial_x u (s,0) = 0$, we have 
		\begin{equation} \label{EInhMP}
		u(t,X_t) = u(s,x_0) + \int_{s}^{t}(L^{\delta}u(r,X_r) + \partial_su(r,X_r))dr + \int_{s}^{t}\partial_xu(r,X_r)dW_r.
	\end{equation}


\end{proposition}	
\begin{proof} \
  Without loss of generality, we set $s = 0$.
   Suppose first that $u \in C^1([0,T]; D)$,
   where $D$ is equipped with the norm introduced
   in Remark \ref{RSetD}.
        In this case, the result follows by standard arguments, see
       e.g. Theorem 4.18 e.g. \cite{BandiniRusso_RevisedWeakDir}.
       In the general case we consider again the sequence $(\chi_n)$
       defined in \eqref{eq:chin} and we set
       $ u_n(t,x) := (u \chi_n)(t,x), (t,x) \in [s,T] \times \R_+$.
       By Remark \ref{rem:SMP} (2),  $u_n$ belongs to $C^1([0,T]; D)$ so that
       \eqref{EInhMP} holds for $u_n$.
       Moreover, by the same Remark,
       $u_n$ (resp. $\partial_x u_n, L_x^\delta u_n$) converges to $u$ (resp.
 $\partial_{x}u, L_x^\delta u$)
       uniformly on compact subsets
       of $[0,T] \times \R_+$. It follows 
     by standard arguments that 
     \eqref{EInhMP} also holds for
 $u \in C^{1,2}([s,T] \times \R_+;\R_+)$
  such that $\partial_x u (s,0) = 0$.
\end{proof}

\subsection{Transition semigroup of Bessel process}

$ $

In this section, we present some properties associated to the Bessel process marginal laws. We refer the reader to e.g. Chapter 6 and Appendix A of
\cite{Jean}, for more details.
The reader may also consult Chapter XI 
in \cite{Yor} for related properties
and in particular the discussion after Definition (1.9) of the same book.

 We denote by $\shb_b(\R_+)$ as the linear space of Borel bounded
      functions $f:\R_+ \rightarrow \R$.
The marginal distributional law of
$X_t, t > 0,$ starting from $x\geq 0$ is given by
\begin{equation}\label{density1}
  \E[f(X_t)] = 
  \int_{\R_+}p^{\delta}_t(x,y)f(y)dy,
\end{equation}
for $f \in \shb_b(\R_+)$. Here,  
\begin{equation}\label{Bd}
	p^{\delta}_t(x,y):=\dfrac{y}{t}\left(\dfrac{y}{x}\right)^{\nu}\exp\left(-\dfrac{x^2 + y^2}{2t}\right)I_{\nu}\left(\dfrac{xy}{t}\right), \ \ t, x,y>0,
      \end{equation}
      and for $t,y > 0$ and $x=0$, 
\begin{equation}\label{Bd0}
  p^{\delta}_t(0,y) := 2^{-\nu}t^{-(\nu + 1)}[\Gamma(\nu + 1)]^{-1}y^{2\nu + 1}
  \exp\left(-\dfrac{y^2}{2t}\right),
\end{equation}
where $\nu = \frac{\delta}{2}-1$. Here, $I_{\nu}$ is the so-called modified
Bessel function (see \cite{abra} p.g. 374 and \cite{watson} p.g. 77) and $\Gamma$ is the Gamma function, see \cite{abra} p.g 255.

\begin{remark}\label{r1.0}
  For $t,x,y>0$, one can easily check that $p^{\delta}_t(x,y)y^{1 - \delta} = p^{\delta}_t(y,x)x^{1-\delta}$.
\end{remark}

\begin{remark}\label{RBd}
  For $0 < t$ and  $ y > 0,$
  $x\mapsto p^{\delta}_t(x,y)$ is of class
  $C^1( (0,+\infty))$ (therefore absolutely continuous) and
\begin{equation}\label{dxBd}
 	\partial_{x}p^{\delta}_{t}(x,y) = \frac{1}{2t}(p^{\delta + 2}_{t}(x, y) - p^{\delta}_{t}(x, y)).
      \end{equation}
      This is a consequence of the recursive property $\partial_x I_{\nu}(x) = I_{\nu + 1}(x) + \frac{\nu}{x}I_{\nu}(x)$ (see (9.6.26) in Chapter 9 of \cite{abra} or Section 3.71 of \cite{watson}). In fact  \eqref{dxBd}
     comes out differentiating
     \eqref{Bd}.
      \end{remark}
     
     For each $ f \in \shb_b(\R_+)$,  we denote
\begin{equation} \label{density}
  P^{\delta}_t[f](x) :=  \left
    \{
\begin{array}{ccc}
   f(x) &;&  t = 0 \\
  \int_{\R_+}p^{\delta}_t(x,y)f(y)dy &;& t > 0.
\end{array}
\right.
  \end{equation}    

 
 \begin{proposition}\label{wellP}
  For a given $t\ge 0$, the mapping
    $P^{\delta}_t:\shb_b(\R_+) \cap  L^2(d\mu) \subset L^2(d\mu) \rightarrow L^2(d\mu)$,
    given by $\eqref{density}$,
      has the contraction property, see item (3) of Definition \ref{d2}.
\end{proposition}

\begin{proof}
Let $f \in \shb_b(\R_+) \cap L^2(d\mu) $.
  When $t = 0$ the statements are obviously true. So we suppose $t >0$.
   Since $p^{\delta}_t$ is a probability density, we can apply Jensen's inequality and deduce that $(P^{\delta}_t[f](x))^2\leq P^{\delta}_t[f^2](x)$ for every $x \in \R_+$. By integrating on both sides of that inequality, using expression \eqref{density} and applying Tonelli's theorem, we get   
	\begin{eqnarray*}
          ||P^{\delta}_t[f]||^2_{L^2(d\mu)} &=& \int_{\R_+}(P^{\delta}_t[f](x))^2\mu(dx) \leq \int_{\R_+}\int_{\R_+}p^{\delta}_t(x,y)f^2(y)dy\mu(dx) \\
         &=& \int_{\R_+}f^2(y)\int_{\R_+}p^{\delta}_t(x,y)\mu(dx)dy  \\ &=& \int_{\R_+}f^2(y)\int_{\R_+}p^{\delta}_t(x,y)y^{1-\delta}\mu(dx)\mu(dy).
	\end{eqnarray*}
	
	By Remark \ref{r1.0}, the latter expression is equal to 
	\begin{eqnarray*}
          \int_{\R_+}f^2(y)\int_{\R_+}p^{\delta}_t(y,x)x^{1-\delta}\mu(dx)\mu(dy)
          &=& \int_{\R_+}f^2(y)\int_{\R_+}p^{\delta}_t(y,x)dx\mu(dy) \\
          &=& \int_{\R_+}f^2(y)\mu(dy) = ||f||^2_{L^2(d\mu)}.
	\end{eqnarray*}
\end{proof}

Since $\shb_b(\R_+) \cap  L^2(d\mu)$ is dense in $L^2(d\mu)$,
Proposition \ref{wellP}  allows us to extend continuously $P_t^\delta$
to $L^2(d\mu)$. We  will of course adopt the same notation
for that extension.
\begin{proposition}\label{c1}
  For each $t \ge 0$, 
    $P^{\delta}_t:    L^2(d\mu) \rightarrow L^2(d\mu)$ has
  the contraction property and, in particular, it is
continuous with respect to the $L^2(d\mu)$-norm. Moreover, for every $f \in L^2(d\mu)$
  \begin{equation} \label{ExtEquality}
P^\delta_t[f](x) = \int_{\R_+} p^\delta_t(x,y) f(y) dt, \ x \ {\rm a.e.}
\end{equation}
and, for every $f, g \in L^2(d\mu)$ and $t \ge 0$, 
\begin{equation} \label{Psymmetry}
  \langle P^\delta_t[f], g \rangle_{L^2(d\mu)} =
  \langle f,  P^\delta_t[g] \rangle_{L^2(d\mu)}.
  \end{equation}
\end{proposition}
\begin{proof}
  The first part holds true because $\shb_b(\R_+) \cap  L^2(d\mu)$ is dense
  in $L^2(d\mu)$ and the fact that $t \mapsto P_t[f]$
  is continuous. Concerning 
  \eqref{ExtEquality}, let us fix $f \in L^2(d\mu)$. Decomposing $f$ into $ f = f^+ - f^-$,
  we are allowed to suppose that $f$ is non-negative.
  We remark indeed that $f^+$ and $f^-$ still belong to $L^2(d\mu)$.
  For each $m > 0$ we denote
  $$ f^m(x) = (f(x) \vee (-m)) \wedge m, \ x \ge 0.$$
  Since $f^m$ is a bounded Borel function,
  we have
  \begin{equation} \label{ExtEqualityTrunc}
    P^\delta_t[f^m](x) = \int_{\R_+} p_t^{\delta}(x,y) f^m(y) dy, \  \forall x >0.
     \end{equation}
Let us fix $x > 0$.
Since $p_t^{\delta}(x,y) dy$ is a Borel probability measure,
letting  $m$ going to infinity,
the right-hand side of
\eqref{ExtEqualityTrunc} converges to 
the right-hand side of \eqref{ExtEquality},
which could theoretically be infinite.

Now each $f^m$ also belongs to $L^2(d\mu)$:
we remark that, for this, $\{0\}$ is not relevant since $\mu(\{0\}) = 0$.
Since $(f^m)$ converges in $L^2(d\mu)$ to $f$,
then $(P^\delta_t[f^m])$ converges in $L^2(d\mu)$ to $P^\delta_t[f]$, there is a subsequence $(m_k)$  such that
$(f^{m_k})$ converges a.e. to $f$  
and  $P^\delta_t[f^{m_k}]$ converges a.e. to $P^\delta_t[f]$. 
Consequently for $x >0$ a.e.
$$ P_t^\delta[f](x) = \lim_{k \rightarrow +\infty}  P^\delta_t[f^{m_k}](x)
=  \lim_{k \rightarrow +\infty}\int_{\R_+} p_t^{\delta}(x,y) f^{m_k}(y) dy =
\int_{\R_+} p_t^{\delta}(x,y) f(y) dy.$$
In particular, for almost all $x \ge 0$, 
the right-hand side of \eqref{ExtEquality}
is finite.

Concerning \eqref{Psymmetry}, let us suppose $t > 0$. By using Remark \ref{r1.0}, we get
\begin{eqnarray*}
	\langle P^{\delta}_t[f],g\rangle_{L^2(d\mu)} &=& \int_{\R_+}g(x)\int_{\R_+}p^{\delta}_t(x,y)f(y)d\mu(x)\\
&=& \int_{\R_+}g(x)\int_{\R_+}p^{\delta}_t(x,y)f(y)y^{1-\delta}d\mu(y)d\mu(x) \\
                                                     &=&\int_{\R_+}g(x)\int_{\R_+}p^{\delta}_t(y,x)f(y)x^{1-\delta}d\mu(y)d\mu(x) \\
  &=&
	\int_{\R_+}f(y)\int_{\R_+}p^{\delta}_t(y,x)g(x)x^{1-\delta}d\mu(x)d\mu(y)\\
	&=& \int_{\R_+}f(y)P^{\delta}_t[g](y)d\mu(y) = \langle f,P^{\delta}[g]\rangle_{L^2(d\mu)}.
\end{eqnarray*}
\end{proof}
We observe that \eqref{Psymmetry} implies that the operators  $(P^{\delta}_t)$ 
are symmetric maps in the sense of item (1) of Definition \ref{d2}, with $H = L^2(d\mu).$

\begin{proposition}\label{c1Bis}
The measure $\mu$ is invariant with respect to $(P^{\delta}_t)$ in the sense that
 for every $f \in   L^1(d\mu) \cap  L^2(d\mu)$ and $t\ge 0$, we have
 \begin{equation} \label{muInvariant}
   \int_{\R_+}  P^\delta_t[f](x) d\mu(x)  =  \int_{\R_+} f(x) d\mu(x).
   \end{equation}
\end{proposition}
\begin{proof}
  Taking into account the definition of $\mu$ and  Remark \ref{r1.0}, 
the left-hand side of  \eqref{muInvariant}  equals
  	\begin{eqnarray*}
          \int_{\R_+}f(y)\int_{\R_+}p^{\delta}_t(x,y) y^{1-\delta}\mu(dy)\mu(dx)
          &=&  \int_{\R_+}f(y)\int_{\R_+}p^{\delta}_t(y,x) x^{1-\delta}
              \mu(dx)\mu(dy)
          \\
          &=& \int_{\R_+}f(y)\int_{\R_+}p^{\delta}_t(y,x) dx \mu(dy) \\
          &=& \int_{\R_+}f(y)\mu(dy),
	\end{eqnarray*}
        since $ \int_{\R_+}p^{\delta}_t(y,x)dx = 1$.

  \end{proof}



	
\begin{remark} \label{RFeller}
The discussion just after Definition 1.9 of \cite{Yor} in Chapter XI
says that the family $(P_t^\delta)$ defines a Feller semigroup
on 
 the space $C_0(\R_+)$ of all real continuous functions
defined on $\R_+$ and vanishing at infinity. The definition of Feller semigroup is given in Definition 2.1 in Chapter III of the same book.
\end{remark}
In particular we have the following.

\begin{lemma}\label{semipro}
  Let $f \in C_0(\R_+)$.
  \begin{enumerate}
    \item For all $t \ge 0$ we have $P_t[f] \in C_0(\R_+)$.
    \item	For $0\le s \le t,$ $P^\delta_{s+t}[f] = P^\delta_s[P^\delta_t[f]]$. 
    \item $\displaystyle\lim_{t\downarrow 0} \sup_x \lvert P^\delta_t [f](x) - f(x) \rvert = 0$.
      \end{enumerate}
    \end{lemma}

    \subsection{$L^\delta$    as restriction of the Bessel infinitesimal generator.}
\
\  

A consequence of Lemma \ref{semipro} is the following.
      
\begin{proposition}\label{timecont}
   For a given $f\in D$, the map $t \mapsto P^{\delta}_t[f]$ is differentiable at $t=0$ with values in  $L^2(d\mu)$. Moreover, $\displaystyle\lim_{t\downarrow 0}\dfrac{P^{\delta}_t[f] - f}{t} = L^{\delta}f$ in $L^2(d\mu)$, for each $f\in D$.
\end{proposition}
\begin{proof}
  Let $f \in C_0(\R_+) \cap L^2(d\mu),$ in particular, this is the case if  $f \in D$. In particular $f$ also belongs to $L^1(d\mu)$.
  \begin{enumerate}
    
  \item We first prove that $t \mapsto P^{\delta}_t[f]$ is continuous at $t= 0$
   with values in $L^2(d\mu)$.
   
   Taking into account
    Lemma \ref{semipro} (3), it is enough to prove
    \begin{equation} \label{L1norm}
      \lim_{t \rightarrow 0} \Vert P^\delta_t[f] - f \Vert_{L^1(d\mu)} = 0.
    \end{equation}
    Without loss of generality, we may suppose $f \ge 0$. Then $P^{\delta}_t[f]\ge 0$ for all $t$, which yields
    \begin{equation}\label{ngtvpart}
    	(P^{\delta}_t[f] - f)^- \leq f.
    \end{equation}
    Indeed, setting $g = P^{\delta}[f]-f$, we clearly get
    \begin{equation*}
      g^- = \max(-g,0) = \max(f-P^\delta_t[f], 0) \le \max(f,0) = f.
     \end{equation*}
    By the invariance of $\mu$ (see Proposition \ref{c1Bis}) and \eqref{ngtvpart}, we have
    \begin{equation}\label{leblim}
    	\int_{\R_+}\lvert P^{\delta}_t[f] - f\rvert d\mu = \int_{\R_+}(P_t^{\delta}[f] - f)d\mu + 2\int_{\R_+}(P^{\delta}_t[f]-f)^- d\mu = 2\int_{\R_+}(P_t^{\delta}[f]-f)^- d\mu.
    \end{equation}
By Lemma \ref{semipro} item (3), we have $\displaystyle\lim_{t\to 0}P^{\delta}_t[f](x) - f(x) = 0$. Hence, by \eqref{leblim}, we conclude the proof using Lebesgue dominated convergence theorem.
    
    

   \item In fact, $t \mapsto P^{\delta}_t[f]$ is continuous also
     for $t = t_0 > 0.$
     Indeed, we recall that by Lemma \ref{semipro} (1),  $P^{\delta}_{t_0}[f] \in C_0(\R_+)$.

 At this point we apply  Lemma \ref{semipro} (2) and
 the continuity at $t= 0$  related to the function
 $P_{t_0} f$,
  which belongs to $C_0(\R_+) $,
 see item (1).
 \item
   Now we prove differentiability at $t=0$. We recall that, for every $x \ge 0$, $Y = X^{0,x}$ fulfills \eqref{lmpBis}
   with $s = 0$.
   Taking the expectation in \eqref{lmpBis} and  dividing  by $t$,
   yields
   $$ \frac{P^\delta_t [f](x) - f(x)}{t} = \dfrac{1}{t}\int_0^t P_r^\delta [L^\delta f](x) dr, \ t \ge 0.$$
  So,
  $$\left\rVert\dfrac{P^{\delta}_t[f] - f}{t} - L^{\delta}f\right\lVert_{L^2(d\mu)} \le \dfrac{1}{t}\int_{0}^{t}||P^{\delta}_s[L^{\delta}f] - L^{\delta}f||_{L^2(d\mu)}ds.$$
  Since $t\mapsto P^{\delta}_t[g]$ is continuous,
with $g = L^{\delta}f$,
  we apply the mean value theorem
 for Bochner integrals
 and take the limit when $t\downarrow 0$ and the proof is concluded.

\end{enumerate}

\end{proof}
\begin{remark}\label{r2}
  As a direct consequence of Proposition \ref{timecont}, 
  the generator of the semigroup $(P^{\delta}_t)$ restricted to $D$ coincides with $L^{\delta}$.
      \end{remark}

\begin{proposition}\label{p2} The following relation holds 
	$$\langle f,L^{\delta}g \rangle_{L^2(d\mu)} = \langle g,L^{\delta}f \rangle_{L^2(d\mu)} = -\dfrac{1}{2}\langle g',f'\rangle_{L^2(d\mu)},$$
for $f,g \in D$. In particular, $- L^{\delta}$ is a symmetric non-negative
        definite operator on $D$.
\end{proposition}
\begin{proof}
By using Remark \ref{r1}, we may apply integration by parts to get 

$$
\langle g,L^{\delta}f\rangle_{L^2(d\mu)} =  \dfrac{1}{2}\langle g,(id^{\delta-1}f')'\rangle = -\dfrac{1}{2}\langle (id^{\delta -1}f'),g'\rangle = -\dfrac{1}{2}\langle f',g'\rangle_{L^2(d\mu)},$$ 
for $f,g \in D$. By exchanging the role of $f$ and $g$, this implies that $L^\delta $ is symmetric. In particular, taking $f = g$, $\langle f,L^{\delta}f\rangle_{L^2(d\mu)} = -\dfrac{1}{2} ||f'||^2_{L^2(d\mu)}$ which means that $L^{\delta}$ is non-positive on $D$.
\end{proof}

	
	
	

\subsection{The dynamics evolution space }

\label{S24}
$ $

We define $\shh$ to be the subspace of absolutely continuous
functions $f\in L^2(d\mu)$ such that there exists a function
$g\in L^2(d\mu)$ such that 

$$f(x) - f(y) = \int_{y}^{x}g(z)dz,$$
for all $x,y \ge 0$. Obviously, if $f \in \shh$ then 
 $f' = g,$ where $f'$ is intended in the sense of distributions. We equip $\mathcal{H}$ with the following norm 
\begin{equation}\label{Hnorm2} 
	\Vert f \Vert^2_{\shh} : = \Vert f\Vert^2_{L^2(d\mu)} +
	\frac{1}{2} \Vert g \Vert^2_{L^2(d\mu)}.
\end{equation}

The proof of the following Proposition is given in the Appendix \ref{apB}.

\begin{proposition}\label{eqH}
  \
	\begin{enumerate}
	\item $D$ is dense in $\shh$.
           \item $D$ is dense in $L^2(d\mu)$.
         \end{enumerate}
\end{proposition}

\begin{proposition}\label{stabH}
  For fixed $t > 0,$ $P^{\delta}_t[f]\in \shh$ for all $f\in L^2(d\mu)$
  and $P^{\delta}_t$ maps $L^2(d\mu)$ into $\shh$ continuously.

\end{proposition}	
\begin{proof}
Let $f \in L^2(d\mu)$.
By Proposition \ref{c1}, we first observe
  \begin{equation} \label{Pcont}
\Vert P_t^\delta [f] \Vert^2_{L^2(d\mu)} \le \Vert f \Vert^2_{L^2(d\mu)}.
\end{equation}
  We prove now that $P_t^\delta [f] \in \shh$.
   For this purpose, let $\varphi:\R_+ \rightarrow \R$ be a smooth function with compact
   support. By Remark \ref{RBd},
   for each $y > 0$, $x \mapsto p^{\delta}_t(x,y)$ is
         absolutely continuous. Then, by using  Proposition \ref{c1}, Fubini's theorem 
         and integration by parts, one can write
  \begin{eqnarray}
           \label{EAC}      
    - \int_{\R_+} P^{\delta}_t[f](x) \varphi'(x) dx
 &=&
          - \int_{\R_+} \left( \int_{\R_+} p^{\delta}_t(x,y)  f(y) dy\right)  \varphi'(x) dx   \\
    &=&
          - \int_{\R_+} \left( \int_{\R_+} p^{\delta}_t(x,y) \varphi'(x)
           dx\right)  f(y) dy  \nonumber  \\ &=& 
 \varphi(0)  \int_{\R_+} p^{\delta}_t(0,y) 
             f(y) dy +
    \int_{\R_+}\left( \int_{\R_+} \partial_x p^{\delta}_t(x,y) \varphi(x)
                                      dx\right) f(y) dy \nonumber \\
    &=& \varphi(0)  \int_{\R_+} p^{\delta}_t(0,y) 
             f(y) dy +
    \int_{\R_+}\left( \int_{\R_+} \partial_x p^{\delta}_t(x,y) f(y)
                                      dy\right) \varphi(x) dx. \nonumber 
  \end{eqnarray}
  We remark that for every function $f \in L^2(d\mu)$, we have $x \mapsto \int_{\R_+} \partial_x p^{\delta}_t(x,y) f(y) dy \in L^2(d\mu)$
  by \eqref{dxBd}. This proves that
            $  x \mapsto P^{\delta}_t[f](x)$ is absolutely continuous and
\begin{equation} \label{EAC1}
            (P^{\delta}_t[f])'(x) = \int_{\R_+}f(y)\partial_xp^{\delta}_t(x,y)dy.
  \end{equation}
We now check that 
 $(P^{\delta}_t[f])' \in L^2(d\mu)$.
		By \eqref{EAC1} and again \eqref{dxBd}, we have 
	\begin{eqnarray} \label{Est2}
\nonumber    \Vert (P_t [f])' \Vert^2_{L^2(d\mu)} &=&
    \int_{\R_+}\left(\int_{\R_+}f(y)\partial_xp^{\delta}_t(x,y)dy\right)^2\mu(dx) \\
    &=& \dfrac{1}{4t^2}\int_{\R_+}(P^{\delta + 2}_t[f](x) - P^{\delta}_t[f](x))^2\mu(dx) \nonumber \\
    &\le &
    \dfrac{1}{2t^2}\left(\int_{\R_+}(P^{\delta+2}_t[f](x))^2\mu(dx) + \int_{\R_+}(P^{\delta}_t[f](x))^2\mu(dx)\right) \\
    &\le& \dfrac{1}{t^2}||f||^2_{L^2(d\mu)}. \nonumber
	\end{eqnarray}
	Therefore, from \eqref{Pcont} and
   \eqref{Est2}, $P^{\delta}_t[f] \in \shh$ and
  \begin{equation} \label{EstNorm}
 \Vert P^\delta_t [f] \Vert^2_\shh \le \left(1 + \dfrac{1}{2t^2}\right) \Vert f \Vert^2_{L^2(d\mu)}.
\end{equation}
   This shows that $P^{\delta}_t$ maps continuously $L^2(d\mu)$ to $\shh$. This concludes the proof.
 
            \end{proof}

\section{Friedrichs extension in the Bessel case}

\label{Fext}

In this section, we adapt and apply the concepts and results presented in Appendix \ref{apA}.

\begin{proposition}\label{p5}
  Let $L^\delta: D \subset L^2(d\mu) \rightarrow L^2(d\mu)$ be the operator defined  in \eqref{op}. Then,
  $-L^{\delta}$ admits the Friedrichs extension (see Definition \ref{Fried})
which we denote by
  $ -L^{\delta}_F$.
\end{proposition}
\begin{proof}
  By Proposition \ref{p2} and Proposition \ref{eqH}, $-L^{\delta}$ is non-negative, symmetric and densely defined on $L^2(d\mu)$. Then, by applying Proposition \ref{A.9} (with $T = -L^{\delta}$  and
  $dom(T) = D$), we conclude it admits the Friedrichs extension $-L^\delta_F$.
\end{proof}
By definition, $-L^\delta_F$ is a non-negative self-adjoint operator. Moreover, we observe that the map $\sqrt{-L^\delta_F}:  dom(\sqrt{-L^\delta_F}) \subset
L^2(d\mu) \rightarrow L^2(d\mu)$
is well-defined in the sense of spectral analysis, see considerations before Proposition
\ref{the1} in Appendix \ref{apA}.

The proof of the following lemma is postponed to Appendix \ref{apB}.
\begin{lemma}\label{l1}
  $dom(\sqrt{-L^{\delta}_F}) = \shh$.
  Moreover $D(\epsilon) = dom(\sqrt{-L^{\delta}_F})$, where $\epsilon$ is the
closed  symmetric form associated with
$-L^{\delta}_F$
  (in the sense of Proposition \ref{the1}).

\end{lemma}

\begin{proposition}\label{symeLF}
  For $f,g \in dom(-L^{\delta}_F)$,
  we have
\begin{equation} \label{symeLFStatement}
  \langle f,L^{\delta}_F g\rangle_{L^2(d\mu)} =  \langle L^{\delta}_F f,g\rangle_{L^2(d\mu)} = -\dfrac{1}{2}\langle f',
  g'\rangle_{L^2(d\mu)}.
  \end{equation}
\end{proposition}
\begin{proof}
Since $- L^{\delta}_F$ is self-adjoint, the first equality is evident. We now check that
\begin{equation} \label{E100}
  \langle f,-L^{\delta}_F g\rangle_{L^2(d\mu)} = \frac{1}{2}\langle f', g'\rangle_{L^2(d\mu)},
\end{equation}
for all $f,g \in dom(-L^{\delta}_F)$. 
By Lemma \ref{l1} and Proposition \ref{the1} item (1) and (3),
we observe that $dom(-L^{\delta}_F) \subset \shh$.
Since $D$ is dense in $\shh$,
we only need to prove \eqref{E100} for $f\in D$ and $g\in dom(-L^{\delta}_F)$. By item (1) of Proposition \ref{eqH}, there exists a sequence $(g_n)$ of functions $g_n\in D$ such that $\displaystyle\lim_n g_n = g$ and $\displaystyle\lim_{n}g'_n = g'$ in $L^{2}(d\mu)$. Taking into account the first equality in
\eqref{symeLFStatement}
 and the fact that $-L^{\delta}_F$ is an extension of $-L^\delta$, by Proposition \ref{p2}, we get

\begin{eqnarray*}
\langle f, - L^{\delta}_F g\rangle_{L^2(d\mu)} &=& \langle -L^{\delta}f, g\rangle_{L^2(d\mu)} = \lim_{n} \langle -L^{\delta}f, g_n\rangle_{L^2(d\mu)}\\
& =& \frac{1}{2}\lim_{n}\langle f', g_n'\rangle_{L^2(d\mu)} = \frac{1}{2}\langle f',g'\rangle_{L^2(d\mu)}.
\end{eqnarray*}  
\end{proof}

From now on $P = (P_t)_{t \ge 0}$ stands for the semigroup with generator $- L^{\delta}_F$, whose existence is guaranteed by Proposition \ref{PA2}, as described in Definition \ref{d3}, taking $H = L^2(d\mu)$. The objective now is to prove that
$P = P^\delta$ on $L^2(d\mu),$ with $P^\delta$ defined in
\eqref{density}, see Corollary \ref{equivP}.

\begin{remark}\label{r4} 
  By Corollary \ref{PA0}, for every $\phi \in dom(-L^{\delta}_F)$ and $t>0$, we have
  $P_t[\phi]\in dom(-L^{\delta}_F)$. Moreover,  $\partial_t P_t[\phi] = L^{\delta}_FP_t[\phi] =
  P_t [L^{\delta}_F\phi]$
  with values in $L^2(d\mu)$.
  In particular, the map $ t \mapsto P_t[\phi]$ from $[0,T]$ to $L^2(d\mu)$ is of class $C^1([0,T], L^2(d\mu))$,
  therefore, absolutely continuous.
\end{remark}

Let $\phi \in D$, in particular $\phi$ belongs to $dom(-L^{\delta}_F)$.
By Remark \ref{r4}, 
$$P_t[\phi] = \phi + \int_{0}^{t}P_s[L^{\delta}_F\phi]ds = \phi + \int_{0}^{t}L^{\delta}_FP_s[\phi]ds.$$
Setting $v(t):= P_t[\phi]$, taking into account the first equality in the statement of Proposition \ref{symeLF} and that 
$L^{\delta} = L^{\delta}_F$ on $D$,  for all $f\in D$, we get
\begin{eqnarray}
\nonumber  \langle v(t),f\rangle_{L^2(d\mu)} &=&  \langle \phi,f\rangle_{L^2(d\mu)} + \int_{0}^{t}\langle L_F^\delta v(s), f\rangle_{L^2(d\mu)}ds\\
\label{FPProv}&=& \langle \phi,f\rangle_{L^2(d\mu)} + \int_{0}^{t}\langle v(s),L^{\delta}_Ff\rangle_{L^2(d\mu)}ds.
\end{eqnarray}


Consequently
\begin{equation}\label{FP}
  \langle v(t),f\rangle_{L^2(d\mu)} =
  \langle \phi,f\rangle_{L^2(d\mu)} + \int_{0}^{t}\langle v(s),L^{\delta}f\rangle_{L^2(d\mu)}ds, \ t \in [0,T], \ \forall f \in D.
\end{equation}  
Let us fix $x\geq 0$ and $\phi \in D$.
Taking the expectation on \eqref{lmpBis} (with $s = 0$), substituting there $f$ with $\phi$,  using \eqref{density}
and the fact that $X_t^{s,x}$ has the same law as $X_{t-s}^{0,x}$, 
we obtain
\begin{equation}\label{Eequation}
  P^{\delta}_t[\phi](x) = \phi(x) + \int_{0}^{t}P^{\delta}_s[L^{\delta}\phi](x)ds.
  \end{equation}
Evaluating the $L^2(d\mu)$-inner product of $P^{\delta}_t[\phi]$  against $f \in D$, taking into account \eqref{Psymmetry} in Proposition \ref{c1},
we obtain
\begin{eqnarray} \label{EPdelta}
\langle P^{\delta}_t[f], \phi \rangle_{L^2(d\mu)} &=&
\langle P^{\delta}_t[\phi], f \rangle_{L^2(d\mu)} = \langle \phi, f\rangle_{L^2(d\mu)} +
     \int_0^t \langle P^{\delta}_s[L^{\delta}\phi], f \rangle_{L^2(d\mu)}ds
 \nonumber \\ &=&
 \langle \phi, f \rangle_{L^2(d\mu)} +
                                                                             \int_0^t \langle L^{\delta}\phi, P^{\delta}_s[f] \rangle_{L^2(d\mu)}ds,
  \forall \phi \in D.
\end{eqnarray}

Setting $u(t) := P^{\delta}_t[f], t \ge 0$,
we see that $u$ also solves \eqref{FP} (substituting $v$ with $u$). Next, we wish to show that for each $t \in [0,T]$ and $\phi \in D$, $P^{\delta}_t[\phi] = P_t[\phi]$. This will be a consequence of the following result. 
\begin{proposition}\label{uni1}   We fix $\phi \in D$.
  There exists at most one function
$v \in L^{1,2}(dt d\mu)$
   that satisfies \eqref{FP}.
  
      \end{proposition}

  \begin{proof}
Let $u,v\in L^{1,2}(dt d\mu)$
    be solutions of \eqref{FP}.
    To show that $w = v - u$  vanishes, we will apply the lemma below, see
    \eqref{inter1}.
  \end{proof}
  \begin{lemma}\label{luni1}
    Let
$w \in  L^{1,2}(dt d\mu).$
    Suppose that for every $f \in D$, we have
	\begin{equation}\label{inter1}
		\langle w(t),f\rangle_{L^2(d\mu)} = \int_{0}^{t}\langle w(s),L^{\delta}f\rangle_{L^2(d\mu)}ds, t \in [0,T],
              \end{equation}
              or
	\begin{equation}\label{inter1Bis}
          \langle w(t),f\rangle_{L^2(d\mu)} =
          \int_{t}^{T}\langle w(s),L^{\delta}f\rangle_{L^2(d\mu)}ds, t \in [0,T].
              \end{equation}
              Then $w \equiv 0$.
\end{lemma}
              \begin{proof} \
                We only suppose \eqref{inter1}
               since, under     \eqref{inter1Bis},
               one would proceed similarly.
   
              \begin{enumerate}
              \item
                We start proving that
                %
                for every $f \in dom (-L^{\delta}_F)$
        \begin{equation}\label{inter2}
          \langle w(t),f\rangle_{L^2(d\mu)} = \int_{0}^{t}\langle w(s),
          L^{\delta}_F f\rangle_{L^2(d\mu)}ds.
	\end{equation}
        Since $L^\delta_F$ extends $L$, by \eqref{inter1},
        and using the fact that $\int_0^t w(s) ds$ exists
        as a Bochner integral with values in $L^2(d\mu),$ we have 

 \begin{equation}\label{inter3bis}
   \langle w(t),f\rangle_{L^2(d\mu)} =
 \int_{0}^{t} \langle w(s), L^{\delta}_F f \rangle_{L^2(d\mu)}
  = \left\langle \int_{0}^{t}w(s) ds,
          L^{\delta}_F f\right\rangle_{L^2(d\mu)},
	\end{equation}
for every  $f\in D$. 

        Since $f \mapsto \langle w(t),f\rangle_{L^2(d\mu)}$ is continuous
        with respect to the $L^2(d\mu)$-norm, 
        then
$$\int_{0}^{t}w(s)ds \in dom((-L_F^{\delta})^*) = dom(-L_F^{\delta}).$$

Consequently, since $-L^\delta_F$ is self-adjoint,  by \eqref{inter3bis}
for every $ f \in D$ we have
\begin{equation}
\label{inter4}
  \langle w(t),f\rangle_{L^2(d\mu)} =
    \left\langle L^{\delta}_F  \int_{0}^{t}w(s)ds, f\right\rangle_{L^2(d\mu)}.
\end{equation}
Since $D$  is dense in $L^2(d\mu)$, \eqref{inter4} holds
for every $f \in L^2(d\mu)$, in particular for any $f \in  dom(-L_F^{\delta})$. Again, being $-L^\delta_F$ self-adjoint, and
again by Bochner integral properties,
we now obtain \eqref{inter2} for every
$f \in  dom(-(L_F^{\delta}))$.

\item Next, we check that for every $t \in [0,T]$,
  
		\begin{equation}\label{inter3}
                  \langle w(t), \Phi(t)  \rangle_{L^2(d\mu)} =
                  \int_{0}^t \langle(w(r), \partial_r\Phi + L^{\delta}_F\Phi)(r)
                  \rangle_{L^2(d\mu)}dr,
              \end{equation}
              for every $\Phi \in C^1([0,T]; dom(-L^{\delta}_F)),$
              where $dom(-L^{\delta}_F)$ is equipped with
              the graph norm $\lVert \cdot \rVert_{dom(-L^{\delta}_F)},$
i.e.
              $$ \lVert f \rVert^2_{dom(-L^{\delta}_F)}: =
              \lVert f \rVert^2_{L^2(d\mu)} + \lVert L^\delta_{F} f \rVert^2_{L^2(d\mu)}.
$$
              Under the norm $\|\cdot \|_{dom(-L^{\delta}_F)}$, since
          $-L^{\delta}_F$ is a closed operator, then $dom(-L^{\delta}_F)$ is a Hilbert space. It will be enough to prove \eqref{inter3} for
\begin{equation} \label{ESpecialForm}
  \Phi(t,x):= l(t)f(x),
\end{equation}
  where $l\in C^1([0,T],\R_+)$
          and 
          $f\in dom(-L^{\delta}_F)$.
          Indeed, by Lemma \ref{l2}, 
      taking $ \hat B= dom(-L^{\delta}_F)$,    
      there exists a sequence $\{\Phi_n; n\ge 1\} \subset C^1([0,T];dom(-L^{\delta}_F))$, of type $\Phi_n(t,\cdot) =
      \sum_k f^n_k l^n_k(t)$, $f^n_k \in dom(-L^{\delta}_F), l^n_k \in C^1([0,T];\R_+)$ such that $\Phi_n \to \Phi$ in $C^1([0,T];dom(-L^{\delta}_F)).$

  \item Let us prove now  \eqref{inter3} for $ \Phi$ of the form \eqref{ESpecialForm}. 
    Integrating
          by parts and using \eqref{inter2}, 
          we get
          \begin{eqnarray*}
            \langle w(t), \Phi(t)\rangle_{L^2(d\mu)}
            &=&  l(t) \langle w(t), f\rangle_{L^2(d\mu)}\\ 
&=&
                 \int_0^t \dot l(r) \langle w(r), f\rangle_{L^2(d\mu)} dr 
            + \int_0^t l(r) \langle w(r), L^\delta_F f \rangle_{L^2(d\mu)}dr \\
            &=&  \int_0^t  \langle w(r), \dot l(r)f\rangle_{L^2(d\mu)} dr +
                \int_0^t  \langle w(r),l(r) L^\delta_F f \rangle_{L^2(d\mu)}dr \\
            &=&     \int_0^t  \langle w(r),
                \partial_r \Phi(r) \rangle_{L^2(d\mu)} + 
                    \int_0^t  \langle w(r),L^\delta_F \Phi(r) \rangle_{L^2(d\mu)}dr,
            \end{eqnarray*}
where $\dot l$ denotes the derivative of $l$. 
            This yields therefore \eqref{inter3}.
          \item We extend \eqref{inter3} for
            $ \Phi \in C^1([0,T]; L^2(d\mu)) \cap C^0([0,T];dom(-L^\delta_F))$.

For such $\Phi$, we set \[\Phi_{\epsilon}(t) = \int_{0}^{t}\dfrac{\Phi((s + \epsilon)\wedge T) - \Phi(s)}{\epsilon}ds + \Phi(0), t \in[0,T].\]

Clearly, $\Phi_{\epsilon} \in C^1([0,T];dom(-L^{\delta}_F))$. By \eqref{inter3} replacing $\Phi$ with $\Phi_{\epsilon}, $ for $t \in [0,T]$, we get
\begin{eqnarray*}
\langle w(t),\Phi_{\epsilon}(t)\rangle_{L^2(d\mu)} &=& \int_{0}^{t}\left\langle w(s),\dfrac{\Phi((s+\epsilon)\wedge T) - \Phi(s)}{\epsilon}\right\rangle_{L^2(d\mu)}ds \\        &+& \int_{0}^{t}\left\langle w(s),\int_{0}^{s}\dfrac{L^{\delta}_F\Phi((r+\epsilon)\wedge T) - L^{\delta}_F\Phi(r)}{\epsilon}dr\right\rangle_{L^2(d\mu)}ds\\ &=:& \Phi_{1,\epsilon}(t) + \Phi_{2,\epsilon}(t).
\end{eqnarray*}
Let $t\in [0,T]$. We observe $\Phi_{\epsilon}(t) \xrightarrow[\epsilon \to 0]{} \Phi(t)$ in $L^2(d\mu)$ and hence, 
$$\langle w(t),\Phi_{\epsilon}(t)\rangle_{L^2(d\mu)} \xrightarrow[\epsilon \to 0]{} \langle w(t),\Phi(t)\rangle_{L^2(d\mu)}.$$

Concerning $\Phi_{1,\epsilon}$, since $\Phi \in C^1([0,T],L^2(d\mu))$,
we first extend $\Phi(s)$ after $T$ so that $\dot \Phi(s)  = \dot{\Phi}(T),$
for $s \ge t$.
Then, using mean value theorem
\[\dfrac{\Phi_{\epsilon}((s +\epsilon)\wedge T) - \Phi(s)}{\epsilon}
=\dfrac{1}{\varepsilon} \int_{s}^{s+\varepsilon} \dot \Phi(r) dr
  \xrightarrow[\epsilon \to 0]{}
  \dot {\Phi}(s),
\]
uniformly  in $s$ in $L^2(d\mu)$. So 
\[\Phi_{1,\epsilon}(t) \xrightarrow[\epsilon \to 0]{} \int_{0}^{t}\langle w(s), \Phi(t)\rangle_{L^2(d\mu)}ds.\]

Concerning $\Phi_{2,\epsilon}$, since $L^{\delta}_F\Phi \in C([0,T],L^2(d\mu))$, then
\[\int_{0}^{s}\dfrac{L^{\delta}_F\Phi((r+\epsilon)\wedge T) - L^{\delta}_F\Phi(r)}{\epsilon}dr \xrightarrow[\epsilon \to 0]{} L^{\delta}_F\Phi(s)\]
uniformly in $s$ in $L^2(d\mu)$. Consequently the result follows.

\item The idea is now to prove  
	\begin{equation}\label{inter2.1}
          \langle w(t),g \rangle_{L^2(d\mu)} = 0,\ \ \text{for all}\
        g \in  D, \ t  \in [0,T],
              \end{equation}
              which  will imply $w \equiv 0$
 since $D$ is dense in $L^2(d\mu)$.


\item To prove  \eqref{inter2.1} we fix $t \in [0,T]$ 
and  we set  $\Phi_t(s)=  P_{t-s}[g]$ for $s \in [0,t]$.
  Since $g \in D \subset dom(-L^\delta_F)$,
  by Remark \ref{r4}, for every $t \in [0,T]$,
$\Phi:[0,t] \rightarrow dom(-L^\delta_F)$
belongs to $C^1([0,t];L^2(d\mu)) \cap C^0([0,t];dom(-L_F^\delta)$.
We remark in particular
that, whenever $g \in D$,  the function
$s \mapsto   L^\delta_F P_s  g = P_s L^\delta_F  g$
is continuous and $$ \partial_s \Phi_t(s) + L^{\delta}_F\Phi_t(s) = 0, \ s \in [0,t]. $$
Using  \eqref{inter3}, we obtain
$$    \langle w(t), g\rangle_{L^2(d\mu)} =   \langle w(t), \Phi_t(t)
\rangle_{L^2(d\mu)}
= \int_0^t  \langle w(s), \partial_s \Phi_t(s) + L^\delta_F\Phi_t(s)\rangle ds = 0.$$
This concludes \eqref{inter2.1}.
\end{enumerate}

\end{proof}

\begin{corollary}\label{equivP}
  $P^{\delta}_t[f] = P_t[f]$ for all $f\in L^2(d\mu)$.
\end{corollary}
\begin{proof}
  Let $t \in [0,T]$. We first show the statement
  for $f \in D$. By the considerations just after the Remark \ref{r4}, the function
  $t \mapsto v(t):= P_t [f]$ solves \eqref{FP}.  
Moreover,  $t \mapsto P^\delta_t [f]$ also solves the same equation
  by \eqref{EPdelta}.
  Then, by Proposition \ref{uni1}  we have $P^{\delta}_t[f] = P_t[f]$
  for every $f \in D$.
  By Proposition \ref{c1}, $P_t^\delta$ is continuous on $L^2(d\mu)$. By Proposition \ref{PA2} and Definition \ref{d2}, 
$P_t$ has the same continuity properties.
Since $D$ is dense in $L^2(d\mu)$
the equality extends to all $f \in L^2(d\mu)$.

\end{proof}

By Lemma \ref{l1}, we have $dom(\sqrt{-L^{\delta}_F}) \times dom(\sqrt{-L^{\delta}_F}) = \shh \times \shh$. Therefore,  
the symmetric closed form  $\epsilon$ corresponding to $- L^{\delta}_F$
as described in Proposition \ref{the1}  
(see \eqref{bsf1}) can be characterized as  $\epsilon:\shh \times \shh\rightarrow \R$.

\begin{remark}\label{r3}
  By Proposition \ref{p2} and Proposition \ref{the1}
  item (4) ( with $T = -L^\delta$), for $u,v \in D$, we have  
\begin{equation} \label{er3}
  \epsilon(u,v)=\dfrac{1}{2}\langle v',u'\rangle_{L^2(d\mu)}.
  \end{equation}
\end{remark}

\begin{remark}\label{r3.2}
By Proposition \ref{eqH} (1) 
   $D$ is dense in $\shh$.
Since $\epsilon$ is closed and 
  \eqref{er3} is continuous on $\shh \times \shh$,
  Remark \ref{r3} implies
\begin{equation} \label{bsf2}
\epsilon (u,v) = \dfrac{1}{2}\langle v',u'\rangle_{L^2(d\mu)}, \ \forall  u,v \in \shh.
\end{equation}
\end{remark}

 We can now rewrite \eqref{Hnorm2} as
 
\begin{equation}\label{Hnorm3}
	\Vert f \Vert^2_\shh = \Vert f\Vert^2_{L^2(d\mu)} + \epsilon(f,f). 
\end{equation} 
In order to prove uniqueness of a mild solution of the semilinear PDE \eqref{pde01}, 
we will make use of the following result.

\begin{proposition} \label{uti} For every $f \in \shh$ and $t > 0$, we have
\begin{equation} \label{equti}
  \Vert (P^\delta_t [f])' \Vert_{L^2(d\mu)} \le \dfrac{1}{\sqrt{t}} \Vert f \Vert_{L^2(d\mu)}.\end{equation}
\end{proposition}
\begin{proof} \
  Proposition \ref{stabH} allows to show that $P^{\delta}_t[f] \in \shh$.
The upper bound follows by Proposition \ref{utibis}
taking into account \eqref{bsf2}.

  \end{proof}

\begin{corollary} \label{cuti} 
We have $P^\delta_t[f] \in \shh$ and
\begin{equation} \label{cutiequation}
  \Vert P^\delta_t [f]\Vert_{\shh} \le \left(1 + \dfrac{1}{\sqrt t}\right) \Vert f \Vert_{L^2(d\mu)}.
  \end{equation}
for every $t>0$ and $f \in L^2(d\mu)$.         
      \end{corollary}
      \begin{proof} \
        Let $f \in L^2(d\mu)$. We recall that $P_t[f]$
        belongs to $\shh$ by
        Proposition \ref{stabH}. It remains to establish the upper bound.
         If $f \in \shh$ the result follows from
         Propositions \ref{uti} and Proposition \ref{wellP} (which states the contraction property for $P^\delta_t$) taking into account \eqref{Hnorm2}.
In order to extend to any $f \in L^2(d\mu)$, we can make use
 of Proposition \ref{eqH} which yields the existence of a
   a sequence $(f_n)$ in $\shh$ (even in $D$)
  converging in $L^2(d\mu)$ to $f$. Now \eqref{cutiequation} holds for $f_n$, which implies that the
  sequence $P^\delta_t [f_n]$ is Cauchy in $\shh$.
  Therefore it converges to some $g \in L^2(d\mu)$.  
  By Proposition \ref{stabH}, $ P^\delta_t$
  is continuous from $L^2(d\mu)$  to $\shh$, then $g = P_t[f]$ and finally
  (\ref{cutiequation}) extends to $f \in L^2(d\mu)$.

        \end{proof}

\section{The linear PDE}

\label{linearSEC}

In the sequel, we fix $g \in L^2(d\mu)$
and $l \in L^{1,2}(dt d\mu)$.
In this section, we present some tools
concerning the linear PDE
\begin{equation}\label{pde02Lin}
	(\partial_t + L^{\delta})u + l = 0, \ \ u(T)\equiv g.
      \end{equation}

We denote by $B$ the Banach space
$B = L^1([0,T]; \shh)$ i.e. the space
of (classes of) strongly (Bochner) measurable
functions $u: [0,T] \rightarrow \shh$ 
 such that
\begin{equation}\label{Bnorm}
	||u||_B := \int_0^T ||u(t)||_{\shh} dt < \infty.
\end{equation} 

We denote
\begin{equation}\label{linop}
  v(t) := P^\delta_{T-t}[g] + \int_{t}^{T}P^\delta_{s-t}[l(s)] ds, \ t \in [0,T].
      \end{equation}
               We recall that, by Corollary \ref{equivP},  $P^\delta = P$ on $L^2(d\mu)$, where $P$ is the semigroup associated
               with $-L^\delta_F$, see Proposition \ref{PA2} and Definition \ref{d2}.

\begin{proposition}\label{stabB}
Let $v$ be the function defined in \eqref{linop}.
  For every $t \in [0,T],$ 
   $v(t)\in L^2(\mu)$
   and $v:[0,T] \rightarrow L^2(d\mu)$ is continuous.
   Moreover $v$ belongs to $B$.
\end{proposition}
    \begin{remark} \label{RBoch}
Recall that since $v:[0,T] \rightarrow L^2(d\mu)$ is continuous, then it is
        Bochner-measurable.
      \end{remark}
\begin{prooff} (of Proposition \ref{stabB}).
  
  We decompose
  $$ v = v_0  + v_1,$$
  where
\begin{eqnarray*}
  v_0(t) &:=& P^\delta_{T-t}[g],  \\
 v_1(t) &:=&  \int_{t}^{T}P^\delta_{s-t}[l(s)] ds, \ t \in [0,T].
\end{eqnarray*}
We first prove the statement for $v$
replaced by $v_0$ and $v_1$.

Concerning $v_0,$ $v_0(t) \in L^2(d\mu)$ for every $t \in [0,T]$. Since $t\mapsto P^{\delta}_t$ is strongly continuous (see Definition \ref{d2} with $H= L^2(d\mu)$), then $t \mapsto v_0(t)$ is continuous. By Corollary \ref{cuti}, for every $t \in [0,T)$, 
we have that $v_0 (t) \in \shh$ and
  $$ \lVert P^{\delta}_{T-t}[g] \rVert_{\shh}dt \leq
\left(1 + \dfrac{1}{\sqrt{T-t}}\right)
  \lVert g\rVert_{L^2(d\mu)}.$$
  
Since 
$$\int_0^T \lVert v_0(t) \rVert_\shh dt \leq \lVert g\rVert_{L^2(d\mu)}\int_{0}^{T}\left(1 + \dfrac{1}{\sqrt{T-t}}\right)dt < \infty,
$$
we then observe $v_0 $ belongs to $B$. This proves the statement for $v$
replaced with $v_0$.
                                              
  Concerning $v_1$, for $ t \in [0,T[$,
   the contraction property of $P^{\delta}$ (see Definition \ref{d2})
tells us immediately that $v_1(t) \in L^2(d\mu)$.
 As far as the continuity is concerned, let $t_0, t \in [0,T]$.
Without loss of generality, we can suppose $t > t_0$.

We have
$v_1(t) - v_1(t_0) = -(I_1(t) + I_2(t))$, where
\begin{eqnarray*}
	&I_1(t) = \int_{t_0}^{t}P^{\delta}_{s-t_0}[l(s)]ds \\
	&I_2(t) = \int_{t}^{T}(P^{\delta}_{s-t_0} - P^{\delta}_{s-t})
   [l(s)]ds.
\end{eqnarray*}
As far as $I_1$ is concerned, the contraction property of $P^{\delta}$
gives
\begin{equation*}
  \|I_1(t)\|_{L^2(d\mu)} \leq \int_{t_0}^{t}\|P^{\delta}[l(s)]\|_{L^2(d\mu)}ds \leq \int_{t_0}^{t}\|l(s)\|_{L^2(d\mu)} ds.
\end{equation*}
Then, $\displaystyle \lim_{t\to t_0} \|I_1(t)\|_{L^2(d\mu)} = 0$.

 Concerning $I_2,$ for $s > t$, the semigroup property implies $P^{\delta}_{s-t_0} - P^{\delta}_{s-t} =  P^{\delta}_{s-t}[P^{\delta}_{t-t_0} - I]$.
 By the contraction property

\begin{equation*}
  \|I_2(t)\|_{L^2(d\mu)}\leq \int_{0}^{T}\|(P^{\delta}_{t-t_0} - I)[l(s)]\|_{L^2(d\mu)}ds.
\end{equation*}
Since $(P^\delta_t)$ is strongly continuous (see item (4) in Definition \ref{d2}), then for every $s\in[0,T[,$
we have
$\displaystyle \lim_{t\to t_0}\|(P^{\delta}_{t-t_0} - I)[l(s)]\|_{L^2(d\mu)} = 0$.
Besides, by contraction property, for almost all $s \in [0,T]$,
\[\|(P^{\delta}_{t-t_0} - I)[l(s)]\|_{L^2(d\mu)} \leq 2 \|l(s)\|_{L^2(d\mu)}.
\]
Since $l$ belongs to $L^{1,2}(dtd\mu)$, then we may apply Lebesgue dominated convergence theorem to state $\displaystyle \lim_{t\to t_0}\|I_2(t)\|_{L^2(d\mu)} = 0$.
This concludes the proof of the continuity of $v_1$ on $[0,T]$
with values in $L^2(d\mu).$

It remains to show that $v_1  \in B$.
By Proposition \ref{stabH}, $P^{\delta}_{s-t}[l(s)] \in \shh$.
By a well-known inequality for Bochner integrals, we have
\begin{eqnarray}
  \nonumber \int_0^T \Vert  v_1 (s) \Vert_\shh ds &\le&  
\nonumber                                                        \int_{0}^{T}\int_{t}^{T}\lVert P^{\delta}_{s-t}[l(s)]\rVert_{\shh}dsdt \\ &=&
\nonumber     \int_{0}^{T}\int_{0}^{s}\lVert P^{\delta}_{s-t}[l(s)]\rVert_{\shh}dtds\\ \nonumber
\nonumber                                                  &\leq& \int_{0}^{T}\int_{0}^{s}\left(1+\dfrac{1}{\sqrt{s-t}}\right)\lVert l(s)\rVert_{L^2(d\mu)}dtds \\
\label{EA1} &=& \int_{0}^{T}\lVert l(s)\rVert_{L^2(d\mu)}\int_{0}^{s}\left(1+\dfrac{1}{\sqrt{t}}\right)dtds.
\end{eqnarray}
Also
\[
\int_{0}^{s}\left(1+ \dfrac{1}{\sqrt{t}}\right)dt = s + 2\sqrt{s} \leq T + 2\sqrt{T}.
\]
Since $l \in L^{1,2}(dtd\mu)$, then the right-hand side of \eqref{EA1} is finite and hence $v_1 \in B$. This concludes the proof.
\end{prooff}

In the next section, in order to connect
weak and mild solutions of our backward Kolmogorov-type PDE, we need the following lemma.


\begin{lemma} \label{Linspection}
  Let $g \in L^2(d\mu)$
  and $l \in L^{1,2}(dt d\mu)$.
  Let $v$ be the function defined in \eqref{linop}.
        For every $\phi \in D$, we have 
 \begin{equation}\label{wsnlTer}
    \langle v(t),\phi\rangle_{L^2(d\mu)} = \langle g,\phi\rangle_{L^2(d\mu)} + \int_{t}^{T}
    \langle  v(r), L^\delta \phi\rangle_{L^2(d\mu)}dr + \int_{t}^{T}\langle
    l(r),\phi\rangle_{L^2(d\mu)}dr.
	\end{equation}
      \end{lemma}

      \begin{proof}

By linearity we can reduce the problem to two
    separate cases: when $l \equiv 0$ and when $g\equiv 0$.
    \begin{enumerate}
      \item
    Suppose first that $l \equiv 0$.
    Suppose first $g \in D$. By Remark \ref{r4} $v(t) :=  P^{\delta}_{T-t}[g]$
    belongs to $dom(-L^\delta_F)$ for every $t \in [0,T]$
    and
\begin{equation} \label{eq:hom}
  \partial_s v(s) = - L^\delta_F v(s), s \in [0,T].
  \end{equation}
  Moreover
  $ -L_F^\delta v(s) = 
  P^{\delta}_{T-s}[-L_F^{\delta} g], \ s \in [0,T].$
    Consequently, since $- L^{\delta} [g] \in L^2(d\mu)$, Proposition \ref{stabB},
   implies that $t \mapsto -L^{\delta}_F v(t)$ is  continuous
    with values in $L^2(d\mu)$.
Integrating \eqref{eq:hom}, from a generic $t \in [0,T[$ to $T$, we get
$$ g - v(t) = - \int_t^T L^\delta_F v(s) ds.$$
Taking the inner product of previous equality with $\phi \in D$, using the fact that $L^\delta_F$ is symmetric
and $L^\delta_F$ extends $L^\delta$, we obtain \eqref{wsnlTer}.
By the contraction property, one can easily show that $g\mapsto P^{\delta}_{T-t}[g]$ is linear and continuous from
$ L^2(d \mu)$ to $L^{1,2}(dt d\mu)$.
Consequently, since $D$ is dense in $L^2(d\mu)$, then the equality \eqref{wsnlTer} extends
to every $g \in L^2(d\mu)$. 

\item
The next step consists in fixing $g \equiv 0$ and let $l$ be a 
generic element in $L^{1,2}(dt d\mu)$. In this case,
\eqref{linop} is given by
 \begin{equation} \label{eq:vt_ell}
    v(t) =  \int_t^T P^{\delta}_{r-t} [l(r)] dr, \ t \in [0,T].
    \end{equation}
    We define $w:[0,T] \rightarrow L^2(d\mu)$ by
    
    $$ w(t) := v(t) - \int_t^T l(r) dr, \ t \in [0,T].$$
    Let $\phi \in D$. We need to show
    \begin{equation}\label{wsnlQuater}
    \langle w(t),\phi\rangle_{L^2(d\mu)} =   \int_{t}^{T}
    \langle  v(r), L^\delta \phi\rangle_{L^2(d\mu)}dr.
  \end{equation}
  
 By Proposition \ref{c1}, $P^{\delta}$ is symmetric and then we have 
  \begin{eqnarray}
\nonumber  	\langle w(t),\phi\rangle_{L^2(d\mu)} &=& \left\langle \int_{t}^{T}P^{\delta}_{r-t}[l(r)]dr,\phi\right\rangle_{L^2(d\mu)} - \int_{t}^{T}\left\langle l(r),\phi\right\rangle_{L^2(d\mu)}dr \\
\label{E1} &=& \int_{t}^{T}\left\langle l(r),P^{\delta}_{r-t}[\phi]\right\rangle_{L^2(d\mu)}dr - \int_{t}^{T}\left\langle l(r),\phi\right\rangle_{L^2(d\mu)}dr.
  \end{eqnarray}
  
By Remark \ref{r4}, $P^{\delta}_{r-t}[\phi] \in dom(L^{\delta}_F)$ and 
$$P^{\delta}_{r-t}[\phi] = \phi + \int_{t}^{r}P^{\delta}_{r-s}[L^{\delta}_F\phi]ds.$$
Consequently, by using \eqref{eq:vt_ell} and \eqref{E1}, we have
  \begin{eqnarray*}
  	\langle w(t), \phi\rangle_{L^2(d\mu)} &=&  \int_{t}^{T}\langle l(r), \phi\rangle_{L^2(d\mu)}dr + \int_{t}^{T}\left\langle l(r),\int_{t}^{r}P^{\delta}_{r-s}[L^{\delta}_F\phi]ds\right\rangle_{L^2(d\mu)}dr\\
  	&-&\int_{t}^{T}\langle l(r),\phi\rangle_{L^2(d\mu)}dr \\
  	&=&  \int_{t}^{T}\int_{t}^{r}\langle  l(r),P^{\delta}_{r-s}[L^{\delta}_F\phi]\rangle_{L^2(d\mu)}dsdr = \int_{t}^{T}\int_{s}^{T}\langle l(r),P^{\delta}_{r-s}[L^{\delta}_F\phi] \rangle_{L^2(d\mu)}drds \\
  &=&	\int_{t}^{T}\int_{s}^{T}\langle P^{\delta}_{r-s}[l(r)],L^{\delta}_F\phi\rangle_{L^2(d\mu)}drds =
     \int_{t}^{T}\langle v(s),L^{\delta}_F\phi\rangle_{L^2(d\mu)}ds.
  \end{eqnarray*}
Since $L^{\delta}$ is a restriction of $L^{\delta}_F$, then \eqref{wsnlQuater} follows. This concludes the proof.   

\end{enumerate}
\end{proof}
      Lemma \ref{Linspection} shows in particular that $v$ defined in \eqref{linop}
is a weak solution of \eqref{pde02Lin} in the sense of Definition \ref{D42}.


      \section{The non-linear pde}
      \label{NLPDE}


      \medskip
      
Let $f:[0,T]\times\R_+\times \R\times\R\to \R$, $g\in L^2(d\mu)$ and a constant $C > 0$ such that

\begin{equation} \label{Lgrowth}
  \vert f(t,x,u,v)\vert \le
   \vert f_0(t,x) \vert + C (\vert u \vert + \vert v \vert), \ t \in [0,T], x \ge 0, u,v \in \R,
  \end{equation}
  where $f_0 \in L^{1,2}(dtd\mu).$
 
  
We will consider the PDE
 \begin{equation}\label{pde02}
	(\partial_t + L^{\delta})u + f(\cdot, \cdot, u, \partial_xu) = 0, \ \ u(T)\equiv g.
      \end{equation}

\begin{definition} \label{D41}
  $u$ is said to be
  a \textbf{classical solution} of \eqref{pde02} if it is of class
  $C^{1,2}([0,T] \times \R_+,\R)$ (which induces $g = u(T) \in C^2(\R_+,\R)$) such that
  $\partial_xu(\cdot,0) \equiv 0$
  and it satisfies \eqref{pde02} in the strict sense.
\end{definition}
In previous definition, we remark that, for every $t \in [0,T],$
$u(t) \in \shd_{L^\delta}(\R_+)$
which was defined in \eqref{DL2R+}.

Let $u:[0,T] \rightarrow L^2(d\mu)$ be a Bochner integrable
function such that $u(t) \in \shh$ for almost all $t \in [0,T[$.
We also suppose that
\begin{equation} \label{EL2}
  l: r \mapsto f(r,x,u(r),\partial_x u(r))) \in L^{1,2}(dt d\mu).
  \end{equation}
\begin{remark} \label{R42}
  If
  $u \in L^1([0,T];\shh)$, then
  \eqref{EL2} holds true.
  
  \end{remark}
      \begin{definition} \label{D42}
  	We say that $u: [0,T] \times \R_+ \rightarrow  \R$
        is a \textbf{weak solution} of \eqref{pde02} if for all $\phi \in D$
and $t \in [0,T]$
  \begin{equation}\label{wsnl}
    \langle u(t),\phi\rangle_{L^2(d\mu)} = \langle g,\phi\rangle_{L^2(d\mu)} + \int_{t}^{T}\langle u(s),L^{\delta}\phi\rangle_{L^2(d\mu)}ds + \int_{t}^{T}\langle
    f(s,\cdot, u(s),\partial_xu(s)),\phi\rangle_{L^2(d\mu)}ds.
	\end{equation}
      \end{definition}
      In fact the notion of weak solution
can be even
      be formulated under the more general assumption that 
      $(t,x) \mapsto f(r,x,u(r,x),\partial_x u(r,x)))$  belongs to
      $ L^1 ([0,T]; L^2_{\rm loc}(\R_+))$.
      In fact the test functions in $D$ have compact support.
\begin{definition} \label{D54}
We say that $u: [0,T] \times \R_+ \rightarrow  \R$ is a \textbf{mild solution} of \eqref{pde02}
  if
	\begin{equation}\label{mils}
          u(t) = P^{\delta}_{T-t}[g] + \int_{t}^{T}P^{\delta}_{s-t}[f(s,\cdot, u(s,\cdot),\partial_xu(s,\cdot))]ds, \ t \in \ [0,T],
          	\end{equation}
\end{definition}
where the equality (\ref{mils}) holds in $L^2(d\mu)$.
\begin{remark} \label{Rmils}
  By \eqref{EL2}, \eqref{linop}
  and Proposition \ref{stabB},
the right-hand side of  \eqref{mils} is well-defined.
\end{remark}
  
As expected, we now present the following result. 
\begin{proposition} \label{P45}
  Let $u$ be a classical solution of \eqref{pde02}
  such that such that $u \in L^{1,2}(dtd\mu)$ and $u(t) \in \shh$ for almost all $t$.
  We also suppose \eqref{EL2}.
  Then $u$
  is also a weak solution of (\ref{pde02}).
\end{proposition}
\begin{proof} 

  Let
  $t\in [0,T]$.
  Since $u(t)\in \shd_{L^{\delta}}(\R_+)$ then $L^{\delta}u(t) \in C(\R_+)$.
 However  $u(t)$ does not necessarily belong to $D$
 so that $L^\delta u(t)$  does not necessarily belong to $L^2(d\mu)$.
 Integrating in time both sides of \eqref{pde02},  $u$
  fulfills
  $$ u(t,x) = g(x) + \int_t^T  (L^\delta u(r,x) dr + f(r,x,u(r,x),\partial_x u(r,x))) dr, \ \forall (t,x) \in [0,T] \times \R_+.$$
By integrating against $\phi \in D$ with respect to $\mu(dx)$, we get 
  \begin{eqnarray}
\nonumber    \langle u(t),\phi\rangle_{L^2(d\mu)} &=& \langle g,\phi\rangle_{L^2(d\mu)} + \int_{t}^{T}
     \int_{\R_+}L^{\delta} u(s,x) \phi(x) d\mu(x) ds\\
\label{wsnlBis}& +& \int_{t}^{T}\langle f(s,u(s),\partial_xu(s)),\phi\rangle_{L^2(d\mu)}ds.
   \end{eqnarray}
  We remark that, by \eqref{op},
$(s,x)\mapsto L^\delta u(s,x)$ is locally bounded, so
   $\int_t^T \int_K \vert L^\delta u(s,x) \vert d\mu(x) ds < \infty,$
  for every compact $K$ of $\R_+$.

        The result follows if we show that
\begin{equation} \label{eq:WeakDuality}
  \int_{\R_+}  L^{\delta} \ell (x) \phi(x) d\mu(x) = \langle \ell,  L^{\delta} \phi\rangle_{L^2(d\mu)},
  \end{equation}
for every $ \ell \in \shd_{L^\delta}(\R_+)$.
Now \eqref{eq:WeakDuality} holds of course if $\ell \in D$
because  $L^\delta_F$ is  symmetric
and $L^\delta$ is a restriction of $L^\delta_F$.

Let us suppose now
$\ell \in \shd_{L^\delta}(\R_+)$.
By Remark \ref{rem:SMP} (1), there is a sequence $(\ell_n)$ in $D$ such that
$\ell_n, \ell'_n, L^\delta \ell_n$ can be shown to  converge
respectively to $\ell, \ell', L^\delta \ell$ uniformly on compact intervals.
Since \eqref{eq:WeakDuality} holds for $\ell$ replaced by $\ell_n$,
finally we get \eqref{eq:WeakDuality} also for $\ell$. This concludes the proof. 
  \end{proof}
		

\begin{proposition}[Uniqueness for the homogeneous PDE]\label{unihomo}
	\
        The vanishing function $u \equiv 0$
        is the unique weak solution (in the sense of Definition \ref{D42}) of the  homogeneous PDE
	\begin{equation}\label{homopde}
		\left\{
		\begin{array}{l}
                  (\partialt + L^{\delta}) u = 0,\\
			u(T,x) = 0.
		\end{array}
		\right.
	\end{equation}
\end{proposition}
\begin{proof} \
  Let $w$ be a weak solution of \eqref{homopde}.
  By definition, for every $f \in D$, we have
  \eqref{inter1Bis}
  and then the result follows by
  Lemma \ref{luni1}.
  
\end{proof}

\begin{proposition} \label{WeakMild}
  Let $u:[0,T] \rightarrow L^2(d\mu)$ such that
  $u(t) \in \shh$ for almost all $t \in [0,T]$.
 We also suppose that that $r \mapsto f(r,x,u(r),\partial_x u(r)))$  belongs to $L^{1,2}(dt d\mu)$.
  Then $u$ is a weak solution of  \eqref{pde02} if and only if
	it is a mild solution.
      \end{proposition}
\begin{proof}
  If $u$ is a mild solution, setting
  $$ l(s) = f(s,\cdot, u(s), \partial_xu(s)), \ s \in [0,T],
  $$
  Lemma \ref{Linspection} implies that it is also a weak solution.

	Suppose that $u$ is a weak solution. We set
	\begin{equation}
          v(t,\cdot) := P^\delta_{T-t}[g] + \int_{t}^{T}P^\delta_{s-t}[f(s,\cdot,u(s,\cdot),\partial_xu(s,\cdot))]ds.
  	\end{equation} 
        Applying again Lemma \ref{Linspection} we see that $v$ is also a weak solution of \eqref{pde02}. So, by linearity $u-v$ is a weak solution of \eqref{homopde}. By Proposition \ref{unihomo} $u = v$.
\end{proof}

We introduce now the solution map $A$ related to the PDE \eqref{pde02}
in the sense of mild solutions.
In particular, to $u$ belonging to $B$
we associate $Au$ defined by the right-hand side of
\eqref{mils}, i.e.
\begin{equation} \label{eqOpA}
Au(t):=  P^{\delta}_{T-t}[g] + \int_{t}^{T}P^{\delta}_{s-t}[f(s,\cdot, u(s,\cdot),\partial_xu(s,\cdot))]ds, \ t \in \ [0,T],
\end{equation}

By Remark \ref{R42}, \eqref{EL2} is fulfilled 
and hence Remark \ref{Rmils} allows to state that $Au$ is well-defined.
We also have the following.

\begin{proposition}\label{stabBNonLin}
  For each $u \in B$, we have $Au \in B$.
   Moreover, for every $t \in [0,T]$, 
   $Au(t) \in L^2(\mu)$
   and $t\mapsto Au(t)$ is continuous from $[0,T]$ to $L^2(d\mu)$.
\end{proposition}
\begin{proof}

  The result follows by  Proposition \ref{stabB}
  setting 
  $$ l(s) = f(s,\cdot,u(s), u'(s)), s \in [0,T]. $$
 
\end{proof}

Next statement 
concerns a tool for establishing that map $A$ defined in \eqref{eqOpA}
admits a unique fixed point.

\begin{proposition}\label{contracM} Let $S$ be a generic set
 and $M: S \rightarrow S$ be a generic application.

  Suppose that for some integer $n_0 > 1$, $M^{n_0}$ has a unique fixed point $u\in S$. Then $u$ is also the unique fixed point of $M$.
      \end{proposition}
      \begin{proof}

We have that $M^{n_0}(Mu) =M(M^{n_0}u) = Mu$. Since $u$ is the unique fixed point of $M^{n_0}$ then $Mu = u,$ which proves that $u$ is a fixed point of $M$.

We show now uniqueness.
Suppose that there is another fixed point $v \in S$ of $M$.
Then, applying iteratively $n_0 -1$  times we get $M^{n_0}v = v$
and $M^{n_0}u = u$. Since  $M^{n_0}u$ admits a unique fixed point
then necessarily $u =v$.
\end{proof}

We introduce now a reinforced hypothesis on $f$
defined at the beginning of Section \ref{NLPDE}.
Up to now, $f$ 
 was only supposed to fulfill \eqref{Lgrowth}.

\begin{hypo} \label{Hypf}
	\
\begin{enumerate}
\item  $|f(t,x,y_1,z_1) - f(t,x,y_2,z_2)|\leq C(|y_1-y_2| + |z_1-z_2|),$
  for every $t \in [0,T], x \ge 0, y_1, y_2, z_1, z_2 \in \R$.
\item $f_0(t,x):=(t,x) \mapsto f(t,x,0,0) $ belongs to $L^{1,2}(dt d\mu)$.
  \end{enumerate}
\end{hypo}
Indeed Hypothesis \ref{Hypf} implies \eqref{Lgrowth}.

\medskip
      Concerning the final condition $g$ we still make
      the following assumption.
     \begin{hypo} \label{Hypg} 
       $ g \in L^2(d\mu)$.
\end{hypo}
In the proposition below we make use of the family of equivalent norms
parametrized by $\lambda > 0$, $\lVert \cdot\rVert_{B,\lambda}$, given by
        $$\lVert u \rVert_{B,\lambda} := \int_{0}^{T}\exp(\lambda t)\lVert u(t)\rVert_{\shh}dt.
$$
Clearly $ \lVert  \cdot \rVert_{B,\lambda} = \lVert  \cdot \rVert_{B}$ if $\lambda = 0$.

\begin{proposition}\label{somi}
  Suppose that $f$ and $g$ satisfy Hypotheses \ref{Hypf} and \ref{Hypg}.
  Let $A$ be the map  defined in \eqref{eqOpA}.
  Then $A^2$ is a contraction.
\end{proposition}

\begin{proof}

First we are going to show that for
$u, v \in B$ and
$t \in [0,T]$ such that
$u(t), v(t) \in \shh,$ which happens a.e.
\begin{equation}\label{basin}
  ||Au(t) - Av(t)||_\shh \leq  C_T \int_{t}^{T}||u(s) - v(s)||_{\shh} 
  \dfrac{1}{\sqrt{s-t}} ds,
\end{equation}
where $C_T =  \sqrt{2}C (\sqrt{T} + 1)$ and $C$ comes from item (1) of Hypothesis \ref{Hypf}).
Indeed by a classical inequality for Bochner integrals and by Corollary \ref{cuti}
  \begin{eqnarray*}
    ||Au(t) - Av(t)||_\shh
                      &\le& \int_{t}^{T} \Vert P^{\delta}_{s-t}[f_u(s) - f_v(s)]  \Vert_\shh ds \\ 
                             &\leq& \int_{t}^{T}\Vert f_u(s) - f_v(s)\Vert_{L^2(d\mu)}\left(1+\dfrac{1}{\sqrt{s-t}}\right) ds \\
                                                &\leq& C\int_{t}^{T}(||u(s) - v(s)||_{L^2(d\mu)}
                                                 + ||(u(s) - v(s))'||_{L^2(d\mu)})
                                                       \left( 1 +  \dfrac{1}{\sqrt{s-t}}\right) ds \\
    &\leq& C (\sqrt{T} + 1)  \int_{t}^{T}(||u(s) - v(s)||_{L^2(d\mu)}
                                                 + ||(u(s) - v(s))'||_{L^2(d\mu)})
           \dfrac{1}{\sqrt{s-t}} ds \\
     &\leq& 2 C (\sqrt{T} + 1)  \int_{t}^{T} ||u(s) - v(s)||_{\shh}
           \dfrac{1}{\sqrt{s-t}} ds,
     \end{eqnarray*}
  where $f_u(s) = f(s,\cdot,u(s),u'(s))$.
  In the second to last inequality we have used the fact that
  $  1 \le \dfrac{\sqrt{T}}{\sqrt{s-t}}$ and in the last inequality we have
  used the fact that, for $a,b \ge0$
  we have $a + b \le 2 \sqrt{a^2 + \dfrac{b^2}{2}}.$
 This establishes \eqref{basin}.
From \eqref{basin} we deduce 
\begin{eqnarray*}
||A^2u(t) - A^2v(t)||_{\shh} &\leq& C_T  \int_{t}^{T}||Au(s) - Av(s)||_{\shh}
    \dfrac{1}{\sqrt{s - t}}ds \leq\\ &\leq& C_T^2\int_{t}^{T}\int_{s}^{T}||u(r) - v(r)||_{\shh}
    \dfrac{1}{\sqrt{r - s}}dr\dfrac{1}{\sqrt{s - t}}ds \\ &\leq& C_T^2\int_{t}^{T}||u(r) - v(r)||_{\shh}\int_{t}^{r}\dfrac{1}{\sqrt{r - s}}\dfrac{1}{\sqrt{s - t}}dsdr.
\end{eqnarray*}

We know that $$\int_{t}^{r}\dfrac{1}{\sqrt{r - s}}\dfrac{1}{\sqrt{s - t}}ds = \int_{0}^{r - t}\dfrac{1}{\sqrt{r - s -t}}\dfrac{1}{\sqrt{s}}ds =
\beta \left(\dfrac{1}{2}, \dfrac{1}{2}\right) = \pi,$$
where $\beta$
is the usual Beta function (see \cite{abra} section 6.1), so
\begin{equation}\label{basin2}
||A^2u(t) - A^2v(t)||_{\shh} \leq C_T^2\pi \int_{t}^{T}||u(r) - v(r)||_{\shh}dr.
\end{equation}


\eqref{basin2} implies
\begin{eqnarray*}
  \int_{0}^{T}\exp(\lambda t)\lVert A^2u(t) - A^2v(t)\rVert_{\shh}dt &\leq& C^2_T\pi\int_{0}^{T}\exp(\lambda t)\int_{t}^{T}\lVert u(s) - v(s)\rVert_{\shh}dsdt \\ &=& C^2_T\pi\int_{0}^{T}\lVert u(s) - v(s)\rVert_{\shh}\int_{0}^{s}\exp(\lambda(t))dtds \\
   &=& \dfrac{C^2_T\pi}{\lambda}\int_{0}^{T}\lVert u(s) - v(s)\rVert_{\shh}(\exp(\lambda s)-1)ds  \\
  &\le& \dfrac{C^2_T\pi}{\lambda}\int_{0}^{T}\lVert u(s) - v(s)\rVert_{\shh} \exp(\lambda s) ds. 
\end{eqnarray*}
This establishes
\[\lVert A^2u - A^2v\rVert_{B,\lambda} \leq
  \dfrac{C^2_T T \pi}{\lambda}\lVert u - v\rVert_{B,\lambda}.\]
Choosing $\lambda > C^2_T T \pi$ we have that $A^2:B \rightarrow B$ is a contraction with respect to $\lVert \cdot \rVert_{B,\lambda}$.

\end{proof}	

\begin{corollary} \label{CFixedP}
  The map $A$ defined in \eqref{eqOpA}
  has a unique fixed point.
\end{corollary}
\begin{proof}
	It is a direct consequence of Propositions \ref{somi} and \ref{contracM}.
      \end{proof}
We can now state the main theorem of the paper.
      \begin{theorem} \label{MainT}
        Let $f:[0,T] \times \R_+\times \R \times \R \to \R$ and
        $g: \R_+ \rightarrow \R_+$ fulfilling Hypotheses  \ref{Hypf} and \ref{Hypg}.
        Then there exists a unique weak solution 
        $u:[0,T] \rightarrow L^2(d\mu)$ of \eqref{pde02}
belonging to $B$.
      \end{theorem}
      \begin{proof} \
       By Proposition \ref{WeakMild},
       it is equivalent to prove existence and uniqueness of a  mild
       solution $u$.

       Concerning existence,
       Corollary \ref{CFixedP} says that the operator $A$, \eqref{linop}, admits a (unique) fixed point $u \in B$.
       By Proposition \ref{stabB} $u(t) \in L^2(d\mu)$ for every $t \in [0,T]$.
       Concerning uniqueness, let $u$ and $v$ be two mild solutions belonging to $B$.
       By Corollary \ref{CFixedP} $u(t) = v(t)$ for almost all $t \in [0,T]$
       as elements of $\shh$ and, therefore, also
       as elements of $L^2(d\mu)$.
       Since mild solutions are continuous in $L^2(d\mu)$
       the result follows.
        \end{proof}

\appendix
	
\section{Semigroups, self-adjoint operators, closed forms and Friedrichs extension}\label{apA}
\

\vspace{0.5cm}
In this section we recall 
some useful functional analysis results
for the study of parabolic PDEs. We also summarize what we need from
the basic theory of the so called Friedrichs self-adjoint extension of a symmetric positive linear operator.
In this section $H$ denotes a Hilbert space with inner product $\langle \cdot,\cdot\rangle_H$ 
and corresponding norm $\Vert \cdot \Vert_H$

For the two definitions below we refer, e.g., to Section 1.3. of \cite{fuku}.
\begin{definition}\label{d2}
  Let $(P_t)_{t \ge 0}$ be a family of linear operators mapping $H$ into $H$.
  We say that $(P_t)_{t>0}$ is a {\bf symmetric contractive strongly continuous semigroup}
  if the following holds.
	\begin{enumerate}
		\item $\langle P_t [u],v\rangle_H = \langle u, P_t[v]\rangle_H$, $u, v \in H$ (symmetry),
		\item $P_tP_s = P_{t+s}, t,s \ge 0$ (semigroup property),
		\item $\Vert P_t [u]\Vert_H \leq \Vert u\Vert_H$
                  (contraction property),
		\item for all $u \in H$, $\displaystyle \lim_{t\downarrow 0}\lVert P_t[u] - u \rVert_H = 0$ (strong continuity).
	\end{enumerate}
      \end{definition}
      In particular the map
      $t \mapsto P_t [u]$ is continuous from
      $[0,T]$ to $H$, taking into account previous
      items (4) and (2).

\begin{definition}\label{d3}
  The generator, $T$, of a symmetric contractive strongly continuous semigroup,
  $(P_t)$, on  $H$ is defined by
	\begin{equation}\label{gene}
		\left\{
		\begin{array}{l}
			Tu = \displaystyle\lim_{t\downarrow 0}\dfrac{P_tu - u}{t}\\
			dom(T) = \{u \in H, \text{for which } Tu \text{ exists in } H\}.
		\end{array}
		\right.
	\end{equation}
	
\end{definition}
The next proposition is an adaptation of Lemma 1.3.2 in Section 1.3 of \cite{fuku}.

\begin{proposition}\label{PA2}
  Let $T$ be a non-positive definite, self-adjoint linear operator on $H$.
  There exists a unique symmetric, contractive and strongly continuous semigroup, $P = (P_t)_{t\geq 0}$, on $H$ such that
        $T$ is the generator of $P$.
\end{proposition} 

The following statement can be found in Corollary 1.4 in Section 1.1 of \cite{pazy}.

  \begin{corollary}\label{PA0}
  Let $T$ be the generator of a symmetric contractive strongly continuous semigroup $(P_t)$ on  $H$. Then for $u \in dom(T)$, we have
  $P_t[u] \in dom(T)$ and $\partial_t P_t[u] = TP_t[u] = P_t[Tu]$ on $H$. 
\end{corollary}

The following definition can be found in Section 1.1 of \cite{fuku}.

		\begin{definition}\label{d4}
                  A bilinear form
                  $\epsilon:D(\epsilon)\times D(\epsilon)\to \R$ on $H$, where $D(\epsilon)$ is a dense linear subspace of $H$, is called a {\bf symmetric form} if
                  it is a ``inner product'' on $D(\epsilon)$ without the assumption
              	 $$\epsilon(u,u) = 0 \Rightarrow u = 0.$$
	We say that a symmetric form $\epsilon$ is
        {\bf closed} when $D(\epsilon)$ is complete with respect to the metric (norm) generated by the inner product
        $$(u,v) := \epsilon(u,v) + \langle u,v\rangle_H.$$   
		\end{definition}
                Let
      $T: dom(T) \subset H
        \rightarrow H$, where $dom(T)$ is a dense linear subspace
        in $H$, be a non-negative self-adjoint linear map.
 By Theorem 1 of Section 6, Chapter XI of \cite{yosi},
        there is a unique so called spectral resolution
        related to $T$. By means of this one can
        define the maps $\phi(T)$ for any continuous function
        $\phi:\R_+ \rightarrow \R_+$, in particular
        we can take $\phi(x) = \sqrt{x}$.
                We consider the map $\sqrt{T}: dom(\sqrt{T}) \subset H \rightarrow H$,
                where $dom(\sqrt{T})$ and $\sqrt{T}$ are intended 
                 via the spectral resolution, see discussion after Lemma 1.3.1 of \cite{fuku}.
                 
                 The next two propositions come from Theorem 1.3.1 and Corollary 1.3.1 in Section 1.3 of \cite{fuku}.

\begin{proposition}\label{the1}
                  There exists an one to one correspondence between the family of closed symmetric forms $\epsilon$ on $H$ and the family of
                                   non-negative self-adjoint linear operators
                  $T: dom(T) \subset H \rightarrow H$.
                  That equivalence is characterized by the following.
			\begin{enumerate}
				\item $D(\epsilon) = dom(\sqrt{T})$,
                                \item
          \begin{equation}\label{bsf1}
	\epsilon(u,v) = \left\langle \sqrt{T}u,\sqrt{T}v\right\rangle_{H}, 
\quad \forall u,v \in D(\epsilon),
      \end{equation}
                 \item
  $ dom(T) \subset D(\epsilon),$
\item $\epsilon(u,v) = \langle v, Tu\rangle_{H}, u\in dom(T) \
                               \ and\ \ v \in dom(\sqrt{T}).$
			\end{enumerate}
			
                      \end{proposition}
                 
		The next proposition is an adaptation of Lemma 1.3.3 in Section 1.3 of \cite{fuku}.
		\begin{proposition}\label{utibis}
                  Let $T$ be a non-negative definite, self-adjoint operator and $\epsilon$ be the closed form corresponding to
             $T$ as described in Proposition \ref{the1}. Let $(P_t)$ be the unique semigroup
whose generator is $-T$ as described in Proposition \ref{PA2}. We have the following properties.
			\begin{enumerate}
                        \item For $t > 0$
$P_t(H)\subset D(\epsilon) (= dom(\sqrt{T}))$.
 \item $\epsilon(P_t[u],P_t[u]) \leq \dfrac{1}{2t}\lVert u\rVert^2_{H}, u \in D(\epsilon).$
			\end{enumerate}
		\end{proposition}

For the following two definitions and proposition, see \cite{ando}.
\begin{definition}\label{OpComp}
  Let $T_1, T_2$ be two non-negative definite, self-adjoint operators on $H$. We say that $T_2$ is greater than $T_1$, we write $T_2\geq T_1$, if
  the following holds:
	\begin{itemize}
		\item $dom(\sqrt{T_2})\subseteq dom(\sqrt{T_1}),$
		\item $\lVert\sqrt{T_1}u\rVert_H \leq \lVert\sqrt{T_2}u\rVert_H$, $u\in dom(\sqrt{T_2}).$
	\end{itemize}
\end{definition}		
		
\begin{definition}\label{Fried}
  Let $H$ be a Hilbert space and $T$ a symmetric, non-negative linear operator on $H$. The greatest (with respect to the order in Definition \ref{OpComp})
  non-negative self-adjoint extension of $T$, if it exists, is called the {\bf Friedrichs extension}
  of $T$ and represented by $T_F$. Then we say that $T$ admits the
  {\bf Friedrichs extension}.
\end{definition}				
				
		\begin{proposition} \label{A.9}
			Let $H$ be a Hilbert space and $T$ a non-negative definite, symmetric, densely defined (i.e. its domain $dom(T)$ is dense in $H$) linear operator. $T$ admits the Friedrichs extension, $T_F$. Moreover
			$$dom(\sqrt{T_F}) = \left\{f \in H; \exists (f_n), f_n \in dom(T), \lim_{n} f_n = f, \lim_{m,n}\langle f_m -f_n, Tf_n - Tf_m\rangle_H = 0 \right\},$$
                        and $T_F$ is the restriction of $T^*$ on $dom(T^*) \cap dom(\sqrt{T_F})$.
                       
                      \end{proposition}
 \begin{remark} \label{A.10}
$g^* \in dom(T^*)$  (as usual) if and only if  the liner form
$f \mapsto \langle Tf, g^* \rangle_H $ is continuous. 
                     \end{remark}

\section{Technical results}\label{apB}
\

\vspace{0.5cm}

We start introducing an useful sequence
of approximating functions.
 Let
 $\chi\in C^{\infty}(\R,[0, 1])$  
verifying
\begin{equation*}
	\chi(x)=\left\{\begin{array}{l}
		1, x\leq -1,\\
		0, x\geq 0.
	\end{array}\right.
\end{equation*} 
For each $ n \in \N, $ we set $\chi_n(x) := \chi(x-(n+1)), \ x \in \R_+$, so
\begin{equation} \label{eq:chin}
	\chi_n(x)=\left\{\begin{array}{l}
		1, x \leq n,\\
		0, x \geq n+1.
	\end{array}\right.
\end{equation}
In particular the functions  $\chi_n$ have compact support.

\begin{remark} \label{rem:SMP}
Let $(\chi_n)$ be the sequence  defined in \eqref{eq:chin}.
    Each $\chi_n$ is a function with compact support
    such that  in a neighborhood of zero,
    $\chi_n$ (resp.  $\chi_n',  \chi_n'')  $
    is equal to $1$ (resp. vanish).

  \begin{enumerate}
  \item Let $ f  \in \shd_{L^\delta}(\R_+)$.
     Setting $f_n:= f \chi_n,  n \in \N$,
    we have that $f_n \in D,$ where $D$ was defined in \eqref{SetD}.
    Moreover it is easy to show that 
$f_n, f'_n, L^\delta f_n$ converge
respectively to $f, f', L^\delta f$ uniformly on each compact interval.
\item Let
   $u \in C^{1,2}([0,T] \times \R_+;\R)$
  with $\partial_x u (s,0) = 0$.
  In particular  $u: [0,T] \rightarrow  \shd_{L^\delta}(\R_+)$.
  We set $u_n(t,x) := u(t,x) \chi_n(x), \ (t,x) \in [0,T] \times \R_+$.
  Then $t \mapsto u_n(t) \in C^1([0,T];D)$.
        Moreover it is easy to see that
        $u_n, \partial_x u_n, L_x^\delta u_n$ converge
        respectively to $u, \partial_xu, L_x^\delta u$
        uniformly on compact sets of $[0,T] \times \R_+$.
        Here $ L^\delta_x$ means $L^\delta$ acting on the space variable $x$.
        \end{enumerate}
\end{remark}

\begin{prooff} (of Proposition \ref{eqH}).
 

  Let $\chi$ and the sequence $(\chi_n)$ defined in \eqref{eq:chin}.
  Let  $f\in L^2(d\mu) \ ({\rm resp.} f\in \shh)$.
We prove now the two items simultaneously.
  We need
  to approach it by a sequence in $D$ converging in $L^2(d\mu)$ ({\rm resp.} $f\in \shh).$
\begin{itemize}
  \item
    Concerning both items (1) and (2) we reduce first
    to the case when  $f$
    has compact support.
    For this we prove that
    there exists a sequence $(f_n)$ converging to $f$ where $f_n$ has compact support.

We fix now $f_n:=f\chi_n$.
The sequence $(f_n)$ obviously converges to $f$ in $L^2(d\mu)$.
Suppose now that $f \in \shh$ and let $(f_n)$ be the same sequence.
We have that $f_n(x) - f_n(y) = \int_{y}^{x}g_n(z)dz$ with $g_n = g\chi_n + f\chi_n'$,
where $g := f'$ (in the sense of distributions). So, $f_n$ belongs to $\shh$.
Moreover $g\chi_n\to g$ and $f\chi_n'\to 0$ in $L^2(d\mu)$. Therefore $g_n\to g$ in $L^2(d\mu)$
which proves the result.

\item For both points (1) and (2), we reduce now
   to the case when $f$ has a compact support and
   vanishes in a neighborhood of zero.
   For this we prove that for $f\in L^2(d\mu) (f\in \shh)$ with compact support there exists a sequence $(f_n)$ converging to $f$ where $f_n$ has compact support and vanishes at $0$.
   
  Suppose that $f \in L^2(d\mu)$ (resp. $f \in \shh$) has compact support.
   We set $\psi_n(x):= \chi(1 - xn), x \ge 0$, so
\begin{equation*}
	\psi_n(x)=\left\{\begin{array}{l}
		1, x\geq \frac{2}{n},\\
		0, x\leq \frac{1}{n}.
	\end{array}\right.
\end{equation*}
and $f_n:= f\psi_n$.
Clearly, since $f \in L^2(d\mu)$ then the sequence $(f_n)$
converges to $f$ in $L^2(d\mu)$, besides $f_n$ has compact support and is null on $[0,\frac{1}{n})$.
Suppose now $f \in \shh$.
We set $g:= f', g_n:= f_n'$ (in the sense of distributions), so that we have
$g_n = g\psi_n + f\psi_n'$. Clearly $g_n \to g$ in $L^2(d\mu)$
and the result is established.
We have then proved  the existence of a sequence $(f_n)$ 
of functions with compact support
vanishing on a neighborhood of $0$
converging to $f$
in $L^2(d\mu)$ (respectively in $\shh$).

\item
  In order to approach a function $f \in L^2(d\mu)$ (resp. $\shh$)
  with support in $(0, +\infty)$
  by a sequence of functions $(f_n)$ belonging to $D$
  we just proceed by convoluting  $f$
   with a sequence
  of mollifiers with compact support converging to the Dirac measure at zero.
  \end{itemize}
\end{prooff}

\begin{prooff} (of Lemma \ref{l1}).

\begin{itemize}
  \item
First we prove the inclusion $dom(\sqrt{-L^\delta_F}) \subset \shh$.
Let $f\in dom(\sqrt{-L^\delta_F})$.

By Proposition \ref{A.9} with $H = L^2(d\mu)$,
$T = -L^\delta$ and $D = dom(T)$,
there exists a sequence $(f_n)$, $f_n \in D$, such that
$\displaystyle \lim_n f_n = f$ in $L^2(d\mu)$
and
\begin{equation*}
\lim_{n,m}\langle f_m - f_n,L^\delta f_n -L^\delta f_m\rangle_{L^2(d\mu)} = 0.
\end{equation*}
Consequently, by Proposition \ref{p2} 
\begin{equation} \label{Ep2}
 \lim_{n,m}   \frac{1}{2}
 \langle f_n'-f_m', f_n'-f_m'\rangle_{L^2(d\mu)}
 = -  \lim_{n,m}   \langle f_m - f_n,L^\delta f_n -L^\delta f_m\rangle_{L^2(d\mu)} = 0.
  \end{equation}
That implies that $(f_n')$ is Cauchy in $L^2(d\mu)$ which yields the existence
of $l\in L^2(d\mu)$ such that $\displaystyle\lim_n f_n'=l$.
It remains to prove that $l$ is the derivative of $f$ in the
sense of distributions.

Let $\phi \in C^{\infty}_0(\R_+)$.
In particular $x \mapsto \phi(x) x^{1-\delta}$ and $x \mapsto \phi'(x) x^{1-\delta}$
belong to
$L^2(d\mu)$
since $\mu$ is a $\sigma$-finite Borel measure on $\R_+$.
We have
\begin{align*}
  \int_{\R_+}\phi(x)l(x)dx & = \int_{\R_+}\phi(x)l(x)x^{1-\delta}\mu(dx) =\lim_n \int_{\R_+}f_n'(x)\phi(x)x^{1-\delta}\mu(dx) \\ & = \lim_n \int_{\R_+}f_n'(x)\phi(x)dx  = -\lim_n\int_{\R_+}\phi'(x)f_n(x)dx \\ &= -\lim_n\int_{\R_+}\phi'(x)f_n(x)x^{1-\delta}\mu(dx) \\
  &= -\int_{\R_+}\phi'(x)f(x)x^{1-\delta}\mu(dx)  = -\int_{\R_+}\phi'(x)f(x)dx,
\end{align*}
where for the second equality
(resp. for the second to last equality) we use the fact that
 $x \mapsto \phi(x) x^{1-\delta}$
(resp. $x \mapsto \phi'(x) x^{1-\delta}$) belongs to $L^2(d\mu)$.
This proves that $l$ is the derivative of $f$ in the sense of distributions and by consequence $f\in \shh$.
\item
  Now we prove the converse inclusion $\shh \subset dom(\sqrt{-L^\delta_F})$. Let $f\in \shh$. By Proposition \ref{eqH} item (1) there exists a sequence
  of functions $(f_n)$, $f_n\in D$, such that  $\displaystyle\lim_n f_n = f$ and $\displaystyle\lim_n f_n' = f'$ in $L^2(d\mu)$. In particular
  $(f_n')$ is Cauchy. Now again Proposition \ref{p2} yields
  $$  \lim_{n,m}\langle f_n - f_m, L^{\delta} f_n - L^{\delta} f_m\rangle_{L^2(d\mu)} = -\frac{1}{2} \lim_{n,m}\langle f_n' - f_m', f_n' - f_m'\rangle_{L^2(d\mu)} = 0.$$
  By Proposition \ref{A.9}, $f$ is shown to belong to $dom(\sqrt{-L^\delta_F})$.
\end{itemize}
\end{prooff}

\begin{lemma}\label{l2}
  Let $\hat{B}$ be a real Banach space and $\Phi:[0,T]\rightarrow \hat{B}$
of class $C^1$.
                Then there exists a sequence $(\Phi_n)$ of the type
                \begin{equation}\label{l2e1}
	 		\Phi_n(t) = \sum_{k}l^n_k(t)f^n_k, \text{ where } l^n_k\in C^\infty([0,T];\R_+) \text{ and } f^n_k \in \hat B,
	 	\end{equation}
                such that  $\Phi_n \rightarrow \Phi$ and  $\Phi'_n \rightarrow \Phi'$ uniformly,
                i.e. in $C^1([0,T];\hat B)$.

\end{lemma}

\begin{proof}
For each $n$ we consider a dyadic partition of $[0,\,T]$ given by  $t_k =2^{-n} kT$,
	$k \in \{0,..., 2^n\}$, $n \in \N$.
	We define the open recovering of $[0,T]$
	\begin{equation}
		U_k^n =
		\left\{\begin{array}{l}
			\left[t_0, t_1\right)\quad k=0,\\
                         \left(t_{k-1}, t_{k+1}\right)\quad k \in \{1,...,2^{n}-1 \}, \\
                         \left(t_{2^n -1}, T\right] \quad k = 2^n.
		\end{array}\right.
	\end{equation}
	
	For each $n$ we consider a smooth partition $(\varphi_k^n)$ of the unity on $[0,T]$.  In particular, 
	\begin{align*}
		&\sum_{k=0}^{2^n} \varphi_k^n =1,\\
		&\varphi_k^n \geq 0,\quad k \in \{0,..., 2^{n}\},\\
		& {\rm supp} \varphi_k^n \subset U_k^n.
	\end{align*}
        We set $v = \Phi'$.
        We define 
		\begin{align*}
			v_n(t)&:=  \sum_{k=0}^{2^n}v(t_k) \varphi_k^n(t).
		\end{align*}
                     Notice that $v \in C^0([0,T];\hat B)$. Since $v$ is uniformly continuous $v_n \to v$ uniformly.
                     We set now
                     $\Phi_n(t):= \Phi(0) + \int_{0}^{t}v _n(s)ds = \Phi(0) +
                     \sum_{k = 0}^{2^n}l^n_k(t)v(t_k),$
                where $l^n_k(t):= \int_{0}^{t}\varphi^n_k(s)ds$.
                We have
                  $\Phi'_n  = v_n \to v= f'$ uniformly, as we have discussed above.
 Consequently $\Phi_n \to \Phi$ in  $C^0([0,T];\hat B),$ which shows that $\Phi_n \rightarrow \Phi$ in $C^1([0,T];\hat B)$.

              \end{proof}

\bigskip

{\bf ACKNOWLEDGEMENTS.} 
The work of FR was partially supported 
 by the  ANR-22-CE40-0015-01 project (SDAIM). The work of AO was partially supported by the 00193-00001506/2021-12 project (FAPDF).

\bibliographystyle{plain}
\bibliography{../../BIBLIO_FILE/biblio-PhD-Alan}
 
\end{document}